\documentclass[final,leqno]{siamltex}
\usepackage[utf8]{inputenc}
\usepackage[english]{babel}
\usepackage{exscale,times}
\usepackage{graphicx}
\usepackage{cancel}
\usepackage{subfigure}

\usepackage{xcolor}

\usepackage[pdftex,
bookmarks = true,
bookmarksnumbered = true,
pdfpagemode = None,
pdfstartview = FitH,
pdfpagelayout = OneColumn,
colorlinks = true,
urlcolor = blue,
pdfborder = {0 0 0} 
]{hyperref}
\usepackage{tikz, fp}
\usepackage{pgfplots}
\pgfplotsset{compat=newest}
\usetikzlibrary{calc, intersections}
\usetikzlibrary{calc, intersections, patterns,positioning, decorations.pathmorphing, arrows, shapes, matrix,decorations.markings}

\usepackage{amsmath,amssymb,amsfonts}
\usepackage{mathabx}
\usepackage{url}
\usepackage{bbm}

\newtheorem{remark}[theorem]{Remark}

\tikzset{math3d/.style=
{x={(1cm,-0.03cm)}, y={(0.20cm,0.40cm)}, z={(0.00cm,1cm)}} }

\title{Overresolving in the Laplace domain for convolution quadrature methods}

\author{T. Betcke\thanks{
Department of Mathematics, University College London, UK, {\tt t.betcke@ucl.ac.uk}}.\and
N. Salles\thanks{Department of Mathematics, University College London, UK, {\tt n.salles@ucl.ac.uk}. Timo Betcke and Nicolas Salles are supported by Engineering and Physical Sciences Research Council Grant EP/K03829X/1.}\and
W. \'Smigaj\thanks{Simpleware Ltd, Exeter, UK, {\tt w.smigaj@simpleware.com}.}}

\begin{document}

\maketitle

\begin{abstract}
Convolution quadrature (CQ) methods have enjoyed tremendous interest in recent years as an efficient tool for solving time-domain wave problems in unbounded domains via boundary integral equation techniques.  In this paper we consider CQ type formulations for the parallel space-time evaluation of multistep or stiffly accurate Runge-Kutta rules for the wave equation. In particular, we decouple the number of Laplace domain solves from the number of time steps. This allows to overresolve in the Laplace domain by computing more Laplace domain solutions solutions than there are time steps.
We use techniques from complex approximation theory to analyse the error of the CQ approximation of the underlying time-stepping rule when overresolving in the Laplace domain and show that the performance is intimately linked to the location of the poles of the solution operator.
Several examples using boundary integral equation formulations in the Laplace domain are presented to illustrate the main results. 
\end{abstract}




\markboth{T. Betcke, N. Salles, and W. \'Smigaj}{Spectral estimates of the inverse Z-transform arising in convolution quadrature methods.}


\section{Introduction}
Let $\Omega$ be a bounded domain with boundary $\Gamma$. We consider the wave equation in the exterior $\Omega^{+}:=\mathbb{R}^3\backslash\overline{\Omega}$ given by
\begin{equation}
\begin{cases}
\displaystyle \frac{\partial^2 u}{\partial t^2}(t; x) - c^2 \Delta_xu(t; x) =0, &x \in \Omega^{+}, \\
\displaystyle u(0; x) = \frac{\partial u}{\partial t}(0; x) = 0, &x \in\Omega^{+}, \\
u(t; x) = g(t; x), &x \in \Gamma.
\end{cases} \label{eqAcoustic}
\end{equation}

With the rise in massively parallel computing in recent years it has become important not only to achieve parallelism in space for the solution of \eqref{eqAcoustic}, but also to exploit parallelism in time.
One way to achieve this is by a Fourier or Laplace transform of the wave equation. This allows us to solve for a range of frequencies in parallel and to reassemble the time-solution by an inverse transform. Closely related to this approach are space-time parallel convolution quadrature (CQ) type schemes. Consider a sequence of equally spaced discrete-time approximations
$$
u_{d}(t_0;x), u_{d}(t_1;x),u_{d}(t_2;x),\dots
$$
generated by, e.g., a multistep or Runge-Kutta scheme, such that $u_{d}(t_n;x)\approx u(t_n;x)$ for $n=0,1,\dots$ and $t_n=n\Delta t$. We now apply a Z-transform to this sequence and define the function
$$
U_{d}(z;x) := \sum_{n=0}^\infty u_{d}(t_n;x)z^n.
$$
It turns out (see Section \ref{sec:cq_formulation}) that for each evaluation of $U_{d}(z;x)$ for a given value $z$, we need to solve $m$ modified Helmholtz problems with typically complex wavenumber and known boundary data arising from a Z-transform of $g$. For multistep schemes we have $m = 1$. For Runge-Kutta schemes $m$ is the number of stages of the scheme. The time-stepping values $u_{d}(t_n;x)$ can then be recovered by a simple Cauchy integral as
\begin{equation}
u_{d}(t_n;x) = \frac{1}{2\pi i} \int_{\mathcal{C}} \frac{U(z;x)}{z^{n+1}}dz.
\label{eqCauchy}
\end{equation}
\eqref{eqCauchy} can be efficiently evaluated by a trapezoidal rule with $N_f$ discretisation points, making necessary the solution of $m \times N_f$ independent modified Helmholtz problems to recover the time steps.

Convolution quadrature methods were introduced by Lubich in \cite{lubich1988a,lubich1988b,Lubich:1994iq}. In recent years they have seen tremendous interest for the solution of exterior time-domain scattering problems via boundary integral equation formulations, see e.g. \cite{BanjaiS08,chappell2009,Banjai:2010ij,banjai2012,chappell2011,laliena2009}. The application to Maxwell problems is discussed in \cite{chen2012,ballani2013}. A recent excellent overview of the literature on CQ type methods is also contained in \cite{Banjai:2014ep} and \cite{costabel2004}.

In this paper we take a slightly different approach to CQ methods. We do not consider the overall convergence of convolution quadrature methods to the continuous wave equation, but rather ask the question how well convolution quadrature approximates the time steps $u_d(t_n;x)$ generated by the underlying time-stepping scheme. The crucial approximation here is the evaluation of the Cauchy integral in \eqref{eqCauchy} via a trapezoidal rule. Based on classical analyticity results for the solution operator of the Helmholtz equation in the frequency domain we will give precise error bounds for the approximation of \eqref{eqCauchy} as the number of evaluation frequencies $N_f$ tends to infinity. Moreover, the analysis will show how many frequency evaluations will be at least necessary to obtain an acceptable accuracy. As a byproduct the analysis in this paper will give decay estimates of the time-stepping values $u_d(t_n;x)$ similar in flavour to classical energy decay estimates for the continuous solution $u(t;x)$ of \eqref{eqAcoustic}.

In order to turn this convolution quadrature approach into a numerical method, a solver for the modified Helmholtz equation in the Laplace domain is needed. Here, we focus on boundary integral formulations, as they are most frequently used in the context of CQ methods, and we analyse how the spectral properties of different formulations (e.g. integral equation of the first or second kind, combined formulations) influence the rate of convergence of the trapezoidal rule for \eqref{eqCauchy}.
Other types of solvers in the Laplace domain are possible such as finite elements with a PML condition \cite{Berenger:1994eu,Kim:2009us}, and the type of analysis presented in this paper immediately extends to these formulations.

The paper is structured as follows. In Section \ref{sec:cq_formulation} we give an overview of parallel convolution quadrature methods with a particular focus on the role of the underlying Z-transform. In Section \ref{sec:resonances} we discuss the analyticity of the solution operator in dependence of the exterior resonances of a related Helmholtz problem. This is needed for the convergence analysis in Section \ref{sec:convergence}. In Section \ref{formulationsfreqpb} we turn our attention to boundary integral formulations and discuss the influence of the poles of the solution operators for various integral equation formulations on the convergence results. In Section \ref{secNumRes} we present numerical results, including a precise convergence estimate of the CQ approximation of the underlying time-stepping rule in case of a three dimensional trapping domain. We finish with conclusions in Section \ref{sec:conclusions}.

\section{Convolution quadrature as a Z-transform method}
\label{sec:cq_formulation}

In this section we review the CQ method. The derivation is similar to those given in \cite{BanjaiS08,Banjai:2010ij} but focuses explicitly on the representation in terms of a Z-transform and its inversion via Cauchy integrals. To simplify the presentation we rewrite \eqref{eqAcoustic} as a first order system of the form
\begin{align}
\begin{cases}
\displaystyle  \frac{1}{c} \frac{\partial Y(t ; x)}{\partial t} = \mathcal{L} Y(t ; x), & x \in  \Omega^{+}, \\
Y(0 ; x)=0, & x \in \Omega^{+}, \\
BY(t;x) = F(t;x), & x\in\Gamma,
\end{cases}
\label{eqFirstOrder}
\end{align}
where $\displaystyle Y(t ; x) = \begin{bmatrix}u(t;x)\\ \frac{1}{c} \frac{\partial u}{\partial t}(t;x)\end{bmatrix}$,  $ \mathcal{L} = \begin{bmatrix} 0 & I \\ \Delta_x & 0\end{bmatrix}$, $B=\begin{bmatrix}I & 0 \\ 0 & 0\end{bmatrix}$, and $F(t;x) = \begin{bmatrix}g(t;x) \\ 0 \end{bmatrix}$.

We first write the frequency problems to be solved when applying a multistep scheme, and then we will see how to apply a $m$-stages Runge-Kutta scheme \cite{banjai2011runge,banjai2012runge,banjai2011error}.

\subsection{Multistep schemes}
We start by applying a multistep rule to the first order system \eqref{eqFirstOrder}. The general form of the discrete scheme is then
\begin{equation}
\frac{1}{c\Delta t}\sum_{j=0}^n \gamma_{n-j}Y_{d}(t_j;x) =  \mathcal{L} Y_d(t_n;x).
\label{eqMultistep}
\end{equation}
Here, $Y_{d}(t_n;x)$ is the sequence of discrete approximations to $Y(t_n;x)$ generated by the multistep rule, and the $\gamma_{n-j}$ are the coefficients of the multistep rule. For example, in the case of implicit Euler we have $\gamma_0=1$, $\gamma_1=-1$, and $\gamma_j=0$, $j>1$. For  convenience we will always assume that $\gamma$ is an infinite sequence, where all but a finite number of elements (corresponding to the multistep rule) are zero.

 We want to apply the Z-transform to \eqref{eqMultistep}. We use the following definition for the Z-transform $\mathcal{Z}\{X\}$ of a general sequence $\{X_n\}$ with $n\geq 0$:
\begin{equation}
\mathcal{Z}\{X\}(z) := \sum_{n=0}^\infty X_n z^n,\quad z\in\mathbb{C}.
\label{eqZtransform}
\end{equation}
Hence, the elements of the sequence become the Taylor coefficients of the function $\mathcal{Z}\{X\}(z)$.
The inverse transform is given by a Cauchy integral as
\begin{equation}
X_n := \frac{1}{2\pi i}\int_{\mathcal{C}} \frac{\mathcal{Z}\{X\}(z)}{z^{n+1}}dz,
\label{eqInverseZtransform}
\end{equation}
where $\mathcal{C}$ is a contour around $0$ inside the domain of analyticity of $\mathcal{Z}\{X\}(z)$. Typically, we use a circle of radius $0 < \lambda \leq 1$. The following well knwon result holds for the existence of the Z-transform.
\begin{proposition}
Let $\{X_n\}$ be a sequence with $|X_n|\leq Ce^{-\alpha n}$, $C>0$, $\alpha\in\mathbb{R}$. Then the Z-transform of $\{X_n\}$ exists and $\mathcal{Z}\{X\}(z)$ is analytic inside every closed disk around $0$ with radius $\lambda<e^{\alpha}$.
\end{proposition}
\begin{proof}
Let $|z| = \lambda.$ Then
\begin{equation*}
\left\vert \mathcal{Z}\{X\}(x) \right\vert \leq \sum_{n=0}^{\infty} \left\vert X_n z^n \right \vert \leq C \sum_{n=0}^{\infty} \left(  \lambda e^{- \alpha} \right)^{n},
\end{equation*}
which converges if $\lambda < e^{\alpha}$.
\end{proof}

Hence, if the sequence $\{X_n\}$ decays exponentially, then $\mathcal{Z}\{X\}(z)$ is analytic within a disk of radius $\lambda>1$. On the other hand, if the sequence is only bounded or grows exponentially, we require $\lambda <1$.

The Z-transform of \eqref{eqMultistep} is given by
\begin{equation}
\frac{1}{c\Delta t}\sum_{n=0}^\infty\left[\sum_{j=0}^n \gamma_{n-j}Y_{d}(t_j;x)\right]z^n =  \mathcal{L}\sum_{n=0}^\infty Y_{d}(t_n;x)z^n.
\label{eqZtransformYeq}
\end{equation}
Define $\gamma(z)=\sum_{n=0}^\infty\gamma_n z^n$ and $Y_d(z;x)=\sum_{n=0}^\infty Y_d(t_n;x)z^n$. Then, the left-hand side of \eqref{eqZtransformYeq} is a convolution of the Taylor coefficients of $\gamma(z)$ and $Y_d(z;x)$. \eqref{eqZtransformYeq} is therefore equivalent to
$$
\frac{1}{c\Delta t}\gamma(z)Y_{d}(z;x) =  \mathcal{L} Y_{d}(z;x).
$$
Translating into a second order form, we obtain the modified Helmholtz problem
\begin{align}
\begin{cases}
\displaystyle  \left(\frac{\gamma(z)}{c \Delta t}\right)^2 U_d(z;x)-\Delta_x U_d(z;x) = 0,  & x \in  \Omega^{+}, \\
U_d(z;x) = G(z;x), &x \in \Gamma,
\end{cases}
\label{eqModifiedHelmholtz}
\end{align}
where $U_d(z;x)=\sum_{n=0}^\infty u_d(t_n;x)z^n$ and $G(z;x)=\sum_{n=0}^\infty g(t_n;x)z^n$. We still need to define suitable boundary conditions towards infinity. Consider a sphere $\mathcal{S}$ of radius $r_0>0$ surrounding the domain $\Omega$. Then in the exterior of $\mathcal{S}$ the solution of \eqref{eqAcoustic} is outgoing, and the appropriate boundary conditions for \eqref{eqModifiedHelmholtz} are outgoing boundary conditions. For a general Helmholtz problem of the form
\begin{equation}
\label{eq:Helmholtz}
\Delta v(x)+k^2v(x)=0,~x\in\Omega^{+}
\end{equation}
with possibly complex wavenumber $k$, outgoing boundary conditions can be defined by requiring that $v$ can be expanded into a series of the form
\begin{equation}
\label{eq:outgoing}
v(x)=\sum_{n=0}^\infty\sum_{m=-n}^n a_{n,m}h_n^{(1)}(kr)Y_n^{m}(\hat{x})
\end{equation}
for $r=|x|> r_0$. Here, $\hat{x}=x/|x|$, $h_n^{(1)}$ is a spherical Hankel function of the first kind, the $Y_n^m$ are spherical harmonics, and the $a_{n,m}$ are expansion coefficients of the solution $v$ \cite{Kim:2009us}.
In the case of a real wavenumber $k$ \eqref{eq:outgoing} is equivalent to the Sommerfeld radiation condition $\lim_{|x|\rightarrow\infty}|x|\left(\frac{\partial}{\partial x}-ik\right)v(x)=0$, and generalises the Sommerfeld radiation condition to arbitrary complex wavenumbers.

Hence, suitable conditions towards infinity of \eqref{eqModifiedHelmholtz} are given by \eqref{eq:outgoing} with wavenumber
$k:=k_z$, where
\begin{equation}
\label{eq:k_z}
k_z:=i\left(\frac{\gamma(z)}{c\Delta t}\right).
\end{equation}

For each given $z$ we can now evaluate $U_d(z;x)$ by solving the boundary value problem \eqref{eqModifiedHelmholtz} together with outgoing boundary conditions specified above. Once we have computed $U_d(z;x)$, the time-stepping values $u_d(t_n;x)$ are obtained by applying the inverse Z-transform (\ref{eqInverseZtransform}) as
\begin{align}
u_d(t_n;x) = \frac{1}{2\pi i}\int_{|z|=\lambda}\frac{U_d(z;x)}{z^{n+1}}dz, \label{eqInverseZtransformLambda}
\end{align}
where we integrate over a circle around the origin with radius $\lambda$. In order to turn this into a numerical method we need to approximate this contour integral. The natural choice is the trapezoidal rule, since it converges exponentially in the number $N_f$ of integration points for periodic analytic functions. Let $z_k = \lambda e^{2\pi i\frac{k}{N_f}}$ with $k=1,\dots,N_f$.  The trapezoidal rule applied to the above contour integral gives
\begin{equation}
u_d(t_n;x)\approx u_d^{N_f}(t_n;x):=\frac{1}{N_f}\sum_{k=1}^{N_f} \frac{U_d(z_k;x)}{z_k^n}. \label{eqInverseZtransformTrapz}
\end{equation}
Using the fact that $U_d(\overline{z} ; x) = \overline{U_d(z ; x)}$, we do not need to solve problem (\ref{eqModifiedHelmholtz}) for $N_f$ different frequencies but only for half the frequencies (see \cite[Section 4.1]{BanjaiS08}).
 Furthermore, the Z-transform of the boundary data and the inverse Z-transform of the solution $U_d$ can be efficiently evaluated via FFT.

We can summarize the multistep convolution quadrature method in three steps:
\begin{enumerate}
\item Compute $\omega_j = \gamma(z_j)/(c\Delta t)$ for equally distributed points $z_j$ located on the circle with radius $\lambda$ used as contour for the inverse Z-transform to get the wavenumbers for the modified Helmholtz problem.
\item For each wavenumber $\omega_j$, approximate the solution of problem \eqref{eqModifiedHelmholtz} using a boundary integral equation formulation or other method.
\item Perform the inverse Z-transform using \eqref{eqInverseZtransformTrapz} to evaluate the time-domain solution.
\end{enumerate}

\subsection{Runge-Kutta schemes}
In order to apply a $m$-stages Runge-Kutta method to the first order system \eqref{eqFirstOrder}, we introduce the internal stages $\left(V_i\right)_{i=1\ldots m}$. A Runge-Kutta method is defined by the matrix $A = (a_{i,j})_{1 \leq i,j \leq m}$ and the two vectors $b=(b_i)_{1 \leq j \leq m}$ and $c=(c_j)_{1 \leq j \leq m}$ (see Appendix \ref{appendix_rungekutta}). The general form of the discrete scheme applied to \eqref{eqFirstOrder} is then \cite{banjai2011runge}
\begin{align}
\label{RungeKuttaSystem}\begin{cases}
\displaystyle V_{i}(t_n ; x) \displaystyle =   Y_d(t_n ; x) + c \Delta t \sum_{j = 1}^{m} a_{i,j}  \mathcal{L} V_{j}(t_n ; x)& \text{ for } i \in \left\{1, \ldots, m \right\}, \\
\displaystyle Y_d(t_{n+1} ; x ) \displaystyle =  Y_d(t_n ; x) + c \Delta t \sum_{j =1}^m b_j  \mathcal{L} V_{j}(t_n ; x),
\end{cases}
\end{align}
where $\mathcal{L} = \begin{bmatrix} 0 & I \\ \Delta_x & 0\end{bmatrix}$. The third vector, $\left(c_j\right)_{j=1 \ldots m}$, that characterises the Runge-Kutta scheme does not appear at this stage; it will appear later for the evaluation of the right-hand side, see \eqref{eqRKboundarycondition}.

By applying the Z-transform to \eqref{RungeKuttaSystem}, one has
\begin{align}
\label{RungeKuttaSystemZT}\begin{cases}
\displaystyle V_i(z ; x) =   Y_d(z ; x) + c \Delta t  \sum_{j=1}^m a_{i,j}  \mathcal{L} V_j(z ; x), \\
\displaystyle z^{-1} Y_d(z ; x) =   Y_d(z ; x) + c \Delta t \sum_{j = 1}^m b_j  \mathcal{L} V_j(z ; x).
\end{cases}
\end{align}

We have to point out that by using a stiffly accurate Runge-Kutta scheme (that means $a_{m,j} = b_j, j \in \{1, \ldots, m\}$), from \eqref{RungeKuttaSystem} we obtain the equality
\begin{align}
V_m(z; x)  = z^{-1} Y_d(z ; x). \label{laststagetosolution}
\end{align}
From the second expression of \eqref{RungeKuttaSystemZT}, we get
\begin{align}
Y_d(z ; x) = \frac{z}{1-z} c \Delta t \sum_{j = 1}^m b_j  \mathcal{L} V_j(z ; x) \quad \text{for } |z| < 1, z \in \mathbb{C},
\end{align}
which can be used in the first expression of \eqref{RungeKuttaSystemZT},
\begin{align}
V_i(z ; x) = c \Delta t \sum_{j = 1}^m \left(\frac{z}{1-z} b_j + a_{i,j}\right)  \mathcal{L} V_j(z ; x), \quad \text{for } |z| < 1, i = 1 \ldots m.
\end{align}
Taking into account the fact that $V_j$ can be decomposed as  $V_j(z ;x) = \left[ R_j(z ;x), S_j(z ; x)\right]^t$ and $\mathcal{L} = \begin{bmatrix} 0 & I \\ \Delta_x & 0\end{bmatrix}$, we obtain a system of equations of the second order:
\begin{align}
\begin{cases}
\displaystyle R_i(z ;x) = c \Delta t \sum_{j = 1}^m \left(\frac{z}{1-z} b_j + a_{i,j}\right) S_j(z ; x), \\
\displaystyle  S_j(z ; x) =c \Delta t \sum_{\ell = 1}^m \left(\frac{z}{1-z} b_{\ell} + a_{j,{\ell}}\right) \Delta_x R_{\ell}(z ; x).
\end{cases} \notag
\end{align}

By introducing $\mathcal{R}(z;x) = \left( R_1(z;x), R_2(z;x), \ldots, R_m(z;x)\right)$, we can write
\begin{align}
\left( \frac{\Delta(z)}{c \Delta t} \right)^2 \mathcal{R}(z ;x) = \Delta_x \mathcal{R}(z ;x), \label{eqmhVectorR}
\end{align}
where
\begin{align}
\Delta(z) = \left( A + \frac{z}{1 -z} \mathbbm{1} b^t\right)^{-1} \label{eqDefDelta}
\end{align}
with $\mathbbm{1} = \left(1, \ldots, 1\right)^{t} \in \mathbb{R}^m$.

We have to diagonalize $\Delta(z)$ in order to decouple the system of
equations and be able to apply a boundary element method. We assume
for the radius $\lambda$ of the integration contour that $\lambda <
1$. In this case $\Delta(z)$ always exists
\cite{Banjai:2010ij}. However, $\Delta(z)$ may not be diagonalizable
for certain values of $z$ within the unit disk. For example, in the
case of Radau IIa this occurs for $z=3\sqrt{3}-5$ (see
\cite[Prop. $3.4$]{Banjai:2010ij} or Appendix
\ref{appendix_rungekutta}). In Section
\ref{subsec:analyticityrk} we discuss this case in more detail.

Let $\mathbb{P}(z)$ be the matrix of eigenvectors of $\Delta(z)$ and
$\mathbb{D}(z)$ the diagonal matrix containing the associated
eigenvalues such that
\begin{align}
\Delta(z) = \mathbb{P}(z) \mathbb{D}(z) \mathbb{P}^{-1}(z), \text{ and } \mathbb{D}(z) = \mbox{diag}\left(\gamma_1(z), \ldots, \gamma_m(z)\right). \label{eqRKdiagonalised}
\end{align}
Then we get the independent equations
\begin{align}
\left( \frac{\gamma_j(z)}{c \Delta t} \right)^2 W_j(z ;x) = \Delta_x W_j(z ;x) \label{eqRKWj}
\end{align}
with  $\displaystyle W_j = \sum_{\ell = 1}^{m} \left(\mathbb{P}^{-1}(z)\right)_{j,\ell} R_{\ell}(z;x)$.

We still need to define the boundary conditions for the frequency problems.
Since $V_{i}(t_n ; x)$ in \eqref{RungeKuttaSystem} is an internal stage, we have the boundary condition (see \cite[subsection $2.2$]{schadleLL06} and \cite[section $2$]{Banjai:2014ep} for example)
\begin{align}
B V_{j}(t_n ; x) = F(t_n + c_j \Delta t, x), \quad x \in \Gamma, \label{eqRKboundarycondition}
\end{align}
where $c_j$ is the $j$th coefficient of the vector $c$ that defines the Runge-Kutta scheme and $F(t;x) = \begin{bmatrix}g(t;x) \\ 0 \end{bmatrix}$ with $g$ the Dirichlet data of the acoustic problem.

Taking into account \eqref{eqRKboundarycondition} and applying the Z-transform, one has the following boundary condition for $R_j$:
\begin{align}
R_j(z; x) = G_{j}(z;x) := \sum_{n \geq 0} g(t_n + c_j \Delta t; x) z^{n}, \quad x \in \Gamma. \label{eq:defGj}
\end{align}
Finally, the boundary condition for equation \eqref{eqRKWj} writes as
\begin{align}
W_{\ell}(z;x) =  \sum_{j = 1}^{m} \left(\mathbb{P}^{-1}(z)\right)_{\ell, j} G_{j}(z;x), \quad x \in \Gamma. \label{eq:boundaryConditionRK}
\end{align}
The frequency problems to solve are:
\begin{align}
\label{eq:rkfrequencyproblems}
\begin{cases}
\displaystyle \left( \frac{\gamma_i(z)}{c \Delta t} \right)^2 W_j(z ;x) = \Delta_x W_j(z ;x), &x \in \Omega^{+}, \\
\displaystyle W_j(z;x) =   \sum_{\ell = 1}^{m} \left(\mathbb{P}^{-1}(z)\right)_{j, \ell} G_{\ell}(z;x), & x \in \Gamma,
\end{cases}
\end{align}
with radiation conditions for $W_j$ at infinity (see
\eqref{eq:outgoing} and \cite{Kim:2009us}). If we use a stiffly accurate
Runga-Kutta scheme \eqref{laststagetosolution}, we now obtain
\begin{align}
U_d(z;x) = z R_m(z;x) = z \sum_{j = 1}^{m} \left(\mathbb{P}(z)\right)_{m,j} W_j(z;x), \label{eq:RKComputationUd} 
\end{align}
and formulas \eqref{eqInverseZtransformLambda} and \eqref{eqInverseZtransformTrapz} provide respectively the solution of our problem and its approximation by trapezoidal rule.


\section{Scattering poles and analyticity of the Laplace domain problem}
\label{sec:resonances}

Crucial for the analysis of the CQ method presented in Section \ref{sec:cq_formulation} is the analyticity of $U_d(z;x)$ with respect to $z\in\mathbb{C}$. Consider the Helmholtz equation \eqref{eq:Helmholtz} with outgoing boundary data \eqref{eq:outgoing} and given Dirichlet boundary conditions $v=g$ on $\Gamma$.

Then one can define the solution operator $\mathcal{B}(k)$, which maps the boundary data $g$ into a solution $v$ of the associated Helmholtz problem with Dirichlet boundary data. The following result holds for the analyticity of $\mathcal{B}$ (see e.g.  \cite[Section $9.7$, Corollary $7.5$]{taylor2010}).
\begin{theorem}
\label{thm:meromorphic}
The solution operator $\mathcal{B}$ is a meromorphic operator-valued function of $k$. The poles $p_j$, $j=1,2,\dots$ of $\mathcal{B}(k)$ are located in the lower half of the complex plane, that is $\text{Im}\{p_j\}<0$ for all $j$.
\end{theorem}

At the poles $p_j$ the solution operator $\mathcal{B}$ loses injectivity, and there exist exponentially growing outgoing waves that satisfy zero Dirichlet boundary conditions on $\Gamma$. These poles are also called \emph{scattering poles} associated with the Helmholtz problem.

\subsection{Analyticity of multistep schemes}
\label{subsec::anamultistep}
Now consider the solution operator $\mathcal{B}_U(z)$ associated with the modified Helmholtz problem \eqref{eqModifiedHelmholtz} that maps boundary data $G(z;x)$ into the solution $U_d(z;x)$. It follows that $U_d(z;x)=\mathcal{B}(k_z)G(z;x)=\mathcal{B}_U(z)G(z;x)$ with $k_z$ as defined in \eqref{eq:k_z}.
Hence, $\mathcal{B}_U(z)=\mathcal{B}(k_z)$, and from the analyticity of $\mathcal{B}$ with respect to $k_z$ it follows that $\mathcal{B}_U(z)$ is analytic with respect to $z$, since $k_z$ is a polynomial in $z$. We therefore obtain the following result.
\begin{theorem}
\label{thm:bmeromorphic}
The solution operator $\mathcal{B}_U(z)$ is a meromorphic function of $z$. It can only have singularities at values $z_j$ satisfying $p_j=i\left(\frac{\gamma(z_j)}{c\Delta t}\right)$.
\end{theorem}

It follows that $\mathcal{B}_U$ is an analytic function of $z$ in the interior of the disk with radius $\lambda_{\mathcal{B}}$ defined by
\begin{align}
\lambda_{\mathcal{B}} := \min_j \{|z_j|\}. \label{eqDefLambdaB}
\end{align}

Analyticity of the solution operator $\mathcal{B}_{U}$ alone does not guarantee analyticity of $U_d(z;x)$. Since $U_d(z;x)=\mathcal{B}_U(z)G(z;x)$ the radius of analyticity of the boundary data $G$ is crucial. Remember that $G(z;x)=\sum_{n=0}^\infty g(t_n)z^n$. Hence, the radius of analyticity $G$ depends on the rate of decay of the time data $g(t_n;x)$.

Consider boundary data given in the form
$$
g(t_n;x) = e^{-\beta t}\sin^5(2t) f(x)
$$
for some sufficiently smooth function $f$ on $\Gamma$ and $\beta>0$ (see e.g. \cite[Section 6.1]{Banjai:2010ij}). Then the radius $\lambda_G$ of analyticity of $G(z;x)$ is determined by the requirement that
$$
|\mathcal{Z}\left\{g(\cdot;x)\right\}(z)|\leq \sum_{n=0}^\infty e^{-\beta n\Delta t}|f(x)|\lambda_G^n<\infty,
$$
and therefore $\lambda_G<e^{\beta \Delta t}$. Hence, as $\Delta t\rightarrow 0$ the radius of analyticity becomes effectively $\lambda_G=1$.

Another example is an incident wave $u_i$ defined by a Gaussian beam of the form
$$
u_i(t;x)=\cos\left(2\pi\left(t-\frac{d\cdot x}{c}\right)f\right)e^{-\frac{\left(t-t_p-\frac{d\cdot x}{c}\right)^2}{2\sigma^2}},
$$
and $g(t;x) := -u_i(t;x)$. The values $g(t_n;x)$ now decay super-exponentially as $n\rightarrow\infty$. It follows that the associated function $G(z;x)$ is an entire function with $\lambda_G=\infty$.

We note that the above Gaussian beam does not satisfy the initial condition of \eqref{eqAcoustic} for $t=0$, introducing a weak singularity in the solution. In practice this is not relevant if the beam starts sufficiently far away from the obstacle, and therefore the size of the boundary data at the obstacle at $t=0$ is effectively zero in machine precision. A rigorous way to obtain smooth boundary data satisfying the initial conditions is to define the modified data
$$
\tilde{g}(t;x) := \left(1-e^{-\frac{t^2}{2\sigma_w^2}}\right)g(t;x)
$$
with a suitably chosen $\sigma_w$. Then, if the beam is starting sufficiently far away from the obstacle we have $\tilde{g}(t;x)\approx g(t;x)$ with an exponentially small error, once the beam arrives at the obstacle. However, $\tilde{g}(0;x)=\frac{\partial \tilde{g}}{\partial t}(0;x)=0$, satisfying the initial conditions. Furthermore, $\lambda_{\tilde{G}}=\infty$ still holds.

Combining the analyticity results for $\mathcal{B}_U(z)$ and $G(z;x)$, we obtain the following statement for the analyticity of $U_d(z;x)$ with respect to $z$.
\begin{theorem}
\label{thm:u_analytic}
Let $\lambda_U:=\min\{\lambda_{\mathcal{B}},\lambda_G\}$. Then the function $U_d(z;x)$ is analytic with respect to $z$ for all $|z|<\lambda_U$.
\end{theorem}
\begin{proof}
We have $U_d=\mathcal{B}_U(z;x) G(z;x)$. Hence, $U_d$ is analytic with respect to $z$ if both $\mathcal{B}_U(z)$ and $G(z;x)$ are analytic with respect to $z$.
\end{proof}

\subsection{A remark on analyticity for Runge-Kutta schemes}
\label{subsec:analyticityrk}
To compute the radius of analyticity of $U_d$ in the Runge-Kutta case we could exploit the diagonalisation \eqref{eqRKdiagonalised} of $\Delta(z)$ and
using formula \eqref{eq:RKComputationUd} write the solution as
\begin{align}
U_d(z; x) = z R_m(z; x) =  z \sum_{j = 1}^{m} \left( \mathbb{P}(z)\right)_{m,j}  \left(  \mathcal{B}_W^{(j)}(z) \sum_{\ell = 1}^{m} \left(\mathbb{P}^{-1}(z)\right)_{j, \ell} G_{\ell}(z;x) \right),  \label{eq:solOpRK}
\end{align}
where $\mathcal{B}_W^{(j)}(z) = \mathcal{B}(k_z^{(j)})$ is the solution operator related to the wavenumber $k_z^{(j)}=i \frac{\gamma_j(z)}{c \Delta t}, \quad j = 1, \ldots, m$. If $\Delta(z)$ is diagonalisable everywhere for $|z|<1$, then the only singularities are those of the scalar solution operator (assuming the boundary data is sufficiently smooth).

However, this diagonalisation may break down at values of $z$ for which $\Delta(z)$ has a multiple eigenvalue, such as at $z=3\sqrt{3}-5$ for Radau IIa. Hence, in the particular case of Radau IIa this would only give analyticity within the disk of radius $3\sqrt{3}-5$.

In the case of Runge-Kutta methods, instead of a scalar solution operator for a scalar PDE we need to consider the solution operator $\mathbb{B}_{\mathcal{R}}(z)$ for the vector modified Helmholtz equation
\begin{align}
\begin{cases}
\Delta_x \mathcal{R}(z;x)  -  \left( \frac{\Delta(z)}{c \Delta t} \right)^2 \mathcal{R}(z;x) = 0,&x \in \Omega^+, \\
 \mathcal{R}(z; x) = \mathcal{G}(z; x),& x\in\Gamma,
\end{cases} \label{eqVectormodHelm}
\end{align}
for $|z|<1$
with outgoing boundary condition
$$
\mathcal{R}(z; x) = \sum_{n=0}^\infty\sum_{\ell=-n}^n h_n^{(1)}\left(i\frac{\Delta(z)}{z\Delta t}r\right)\mathbf{a}_{n,\ell}Y_n^{\ell}(\hat{x}),
$$
for sufficiently large $r=|x|$.
Here, $\mathbf{a}_{n,\ell}$ is a vector of $m$ coefficients. The matrix function $h_n^{(1)}\left(\frac{\Delta(z)}{z\Delta t}r\right)$
is well defined, since $\Delta(z)$ has no eigenvalue at $0$.

If in a certain domain $D\subset\mathbb{C}$ the matrix valued function $\Delta(z)$ is diagonalisable as $\Delta(z) = \mathbb{P}(z)\mathbb{D}(z)\mathbb{P}(z)^{-1}$ with $\mathbb{P}(z)$, $\mathbb{D}(z)$ and $\mathbb{P}^{-1}(z)$ analytic with respect to $z\in D$
by diagonalisation the solution operator $\mathbb{B}_{\mathcal{R}}$ is analytic if and only if the associated scalar solution operator $\mathcal{B}_U$ is analytic. However, the question about what happens in a neighbourhood of points $z$ for which $\Delta(z)$ is not analytically diagonalisable remains open. In Section \ref{sec:rkresults} we show numerical results that indicate that for Radau IIa the singularity of $\Delta(z)$ at $z=3\sqrt{3}-5$ does not influence the rate of convergence, and that as in the scalar case the rate of convergence is dominated by the singularities of the scalar solution operator.


\section{Convergence of the convolution quadrature method}
\label{sec:convergence}
The convergence results in this section describe how well the approximate solution $u_d^{N_f},$ obtained using $N_f$ frequencies in the Laplace domain, approximates $u_d$, the exact solution of the underlying time-stepping rule. We do not consider the question of convergence of $u_d$ against the exact solution $u$, which depends on well known properties of multistep or Runge-Kutta schemes.

The discretisation $u_d^{N_f}$ is obtained by applying a trapezoidal
rule to the contour integral \eqref{eqInverseZtransformLambda}. The
analysis presented in this section is therefore based on classical
convergence estimates for the trapezoidal rule (see e.g. \cite[Theorem
2.1]{trefethen2013}). However, we have chosen to present the
convergence analysis in detail as it highlights the connections
between the time-domain values $u_d(t_n;x)$ and the analyticity of the
Laplace domain function $U_d(z;x)$.

The results in this section only require the analyticity radius
$\lambda_U$ of the Laplace domain solution. While in principle the
results could therefore also be applied to Runge-Kutta methods, precise
estimates of the analyticity radius are only available for the
multistep case.

Using the analyticity of the frequency solution, we can get the following exact error representation.
\begin{theorem} \label{thm2}
Let $u_d$ and $u_d^{N_f}$ be defined by (\ref{eqInverseZtransformLambda}) and (\ref{eqInverseZtransformTrapz}).  Let $0 < \lambda<\lambda_U$, where $\lambda_U$ is the radius of analyticity of $U_d$ as defined in Theorem \eqref{thm:u_analytic}.
For the error $u_d^{N_f}(t_n; x) - u_d(t_n; x)$ we have
\begin{align}
\label{eqErrorNf}
 u_d^{N_f}(t_n; x) - u_d(t_n; x)  = \sum_{\kappa = 1}^{\infty} \lambda^{\kappa N_f} u_d(t_{n+\kappa N_f}; x).
\end{align}
\end{theorem}
\begin{remark} This error representation is well known in the context of trapezoidal rule approximations of analytic functions. However, it highlights immediately that a CQ approximation is accurate either if $\lambda$ is sufficiently small or if the wave at time step $u_d(t_n;x)$ and all subsequent time steps have already left the area of observation, that is the values of $x$ we are interested in.
\end{remark}
\begin{proof}
We choose $0 < \lambda < \lambda_U$. Then $U_d$ is analytic in the disc of radius $\lambda$, and it can be expanded as a Taylor series
\begin{equation}
\label{eq:udtaylor}
U_d(z; x) = \sum_{n=0}^{\infty} c_n z^n
\end{equation}
with coefficients
\begin{align}
c_n = \frac{1}{2 \pi i} \int_{\left| z \right| = \lambda} \frac{U_d(z; x)}{z^{n+1}} dz = u_d(t_n; x). \label{eqDefCp}
\end{align}
Then, inserting (\ref{eq:udtaylor}) into (\ref{eqInverseZtransformTrapz}), it follows that
\begin{align}
u_d^{N_f}(t_j; x) & = \frac{1}{N_f}\sum_{{\ell}=1}^{N_f} \sum_{n=0}^{\infty} u_d(t_n;x)z_{\ell}^{n-j} = \sum_{\kappa=0}^{\infty} \lambda^{\kappa N_f} u_d(t_{j+\kappa N_f};x) \\
& =u_d(t_j;x)+ \sum_{\kappa=1}^{\infty} \lambda^{\kappa N_f} u_d(t_{j+\kappa N_f};x), \notag
\end{align}
since by aliasing
\begin{equation*}
\sum_{\ell = 1}^{N_f} z_{\ell}^{p-j} = \begin{cases} \lambda^{\kappa N_f}N_f & \text{if } p=j+\kappa N_f, \kappa \in \mathbb{N}, \\ 0  & \text{otherwise.}  \end{cases}
\end{equation*}
\end{proof}

In order to determine the asymptotic rate of convergence of this CQ method, we have to bound the error $\left|u_d^{N_f}(t_n; x) - u_d(t_n; x)\right|.$ We first need to obtain a bound on the time-domain values $u_d(t_n;x)$. We have the following Lemma.
\begin{lemma}
\label{lem:cnestimate}
Let the radius of analyticity $\lambda_U$ of $U_d(z;x)$ be defined as in Theorem \ref{thm:u_analytic}. Then for any $0 <\hat{\lambda}<\lambda_U$ we have
\begin{equation}
\label{eq:cnestimate2}
|u_d(t_n;x)|\leq \max _{|z|=\hat{\lambda}}|U_d(z;x)|\hat{\lambda}^{-n}
\end{equation}
It therefore follows that
$$
|u_d(t_n;x)| = \mathcal{O}\left((\lambda_U-\epsilon)^{-n}\right)
$$
for any $\epsilon>0$ arbitrarily small.
\end{lemma}
\begin{proof}
Since $U_d(z;x)$ is analytic with respect to $z$ inside every closed disk of radius $\hat{\lambda}<\lambda_U$, the value of the integrand in \eqref{eqDefCp} is independent of $\hat{\lambda}$, and we can estimate
$$
|u_d(t_n;x)|=|c_n|\leq \max_{|z|=\hat{\lambda}}|U_d(z;x)|\hat{\lambda}^{-n}.
$$
The second statement follows by choosing $\hat{\lambda}$ arbitrarily close to $\lambda_U$.
\end{proof}
\begin{remark}
The result in Lemma \ref{lem:cnestimate} is reminiscent of classical energy decay estimates for the wave equation with zero Dirichlet conditions and nonzero initial conditions in the exterior of a non-trapping  obstacle. Let $S$ be a sphere of radius $R$ such that $S$ surrounds $\Omega$ and the support of the initial data is contained in~$S$.
Let $\|u(t;\cdot)\|_{E,R}:= \left[\int_{S\backslash\Omega}|\nabla u(t;x)|^2+|u_t(t;x)|^2\right]^{1/2}$ be the local energy in $S\backslash\Omega$ and $\|u(0;\cdot)\|_E$ the total energy of the initial data in $S$.

In \cite{Morawetz:1977gq} it is shown that
\begin{equation*}
\|u(t;\cdot)\|_{E,R}\leq Ce^{-\beta t}\|u(t;\cdot)\|_E
\end{equation*}
for $C,\beta>0$.

\end{remark}
The estimate in Lemma \ref{lem:cnestimate} is an asymptotic estimate as $n\rightarrow\infty$ and does not depend on whether $\Omega$ is trapping or not. It only depends on the location of the resonances and the behaviour of the Dirichlet boundary data. Note also that if $\lambda_U<1$ then Lemma \ref{lem:cnestimate} becomes a growth estimate. This is for example the case if $g$ is exponentially growing in time. We also note that the transient behaviour of $u_d(t_n;x)$ may look rather different, for example in multiple scattering configurations. The transient behaviour depends on the geometry and the evaluation point $x$ of the time-domain solution.

Combining Lemma \ref{lem:cnestimate} and Theorem \ref{thm2}, we can bound the error
$\left|u_d^{N_f}(t_n; x) - u_d(t_n; x)\right|$ as $n\rightarrow\infty$.
\begin{theorem}
\label{thm:uconvergence}
Let $0<\lambda<\lambda_U$. Then
$$
\left|u_d^{N_f}(t_n; x) - u_d(t_n; x)\right| =\mathcal{O}\left(\left(\frac{\lambda_U}{\lambda}-\epsilon\right)^{-N_f}\right)
$$
as $N_f\rightarrow\infty$.
\end{theorem}
\begin{proof}
Let $0<\lambda<\hat{\lambda}<\lambda_U$.
Inserting \eqref{eq:cnestimate2} into \eqref{eqErrorNf}, we obtain
\begin{equation}
\begin{split}
\left|u_d^{N_f}(t_n; x) - u_d(t_n; x)\right| &\leq \sum_{\kappa = 1}^{\infty} \lambda^{\kappa N_f} |u_d(t_{n+\kappa N_f}; x)|\\
& \leq  \max_{|z|=\hat{\lambda}}|U_d(z;x)| \hat{\lambda}^{-n} \frac{\left(\frac{\lambda}{\hat{\lambda}}\right)^{N_f}}{1-\left(\frac{\lambda}{\hat{\lambda}}\right)^{N_f}} =\mathcal{O}\left(\left(\frac{\lambda_U}{\lambda}-\epsilon\right)^{-N_f}\right)
\end{split}
\end{equation}
for any $\epsilon>0$ as $N_f\rightarrow\infty$ since we can choose $\hat{\lambda}$ arbitrarily close to $\lambda_U$.
\end{proof}
\begin{remark}
The analysis shows that we can increase the rate of convergence by choosing $\lambda$ small. However, while the rate of convergence indeed increases, choosing $\lambda$ too small creates numerical instabilities that limit the achievable accuracy, as indicated in \cite[Subsection 4.2]{Banjai:2010ij} and also demonstrated in Figure \ref{figChoiceLambda}.
\end{remark}

\begin{figure}
\centering \begin{tikzpicture}[scale=0.60]
\draw[thick,->] (-2,0) -- (7,0);
\draw[thick,->] (0,-4) -- (0,4);
\draw (-.2749,2.0161)  node[left] {pole};
\draw[red, thick] (2.5,0) circle (2.25cm);
\draw[blue, thick] (2.5,0) circle (3.4311cm);
\draw[thick] (2.5,0) -- (-.2749,2.0161) node[pos=0.4, above, sloped] {$\frac{\gamma({\lambda_U}z)}{c \Delta t}$};
\draw[thick] (2.5,-2.25)  -- (2.5,0) node[midway, right] {$\frac{\gamma(\lambda z)}{c \Delta t}$};
\fill (2.5,0) circle (3pt);
\fill[blue] (-.2749,2.0161) circle (3pt);
\draw (2.5, 3.4311) node[above right] {$\mathcal{C}_{{\lambda_U}}$};
\draw (2.5, -2.25) node[below] {$\mathcal{C}_{{\lambda}}$};
\end{tikzpicture}
\caption{Contour for the contour integral with the backward Euler scheme, closest pole to the contour and $\lambda$ and $\lambda_U$.}
\label{figContour}
\end{figure}
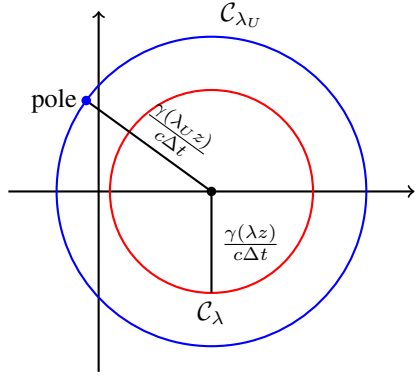

Theorem \ref{thm:uconvergence} is depicted again in Figure \ref{figContour} for the case of backward Euler, where $\gamma(z)=1-z$. The rate of convergence depends on the relative distance of the closest pole to $\mathcal{C}_{\lambda}$ the red circle with radius $\lambda$.

\section{Boundary integral formulations of the frequency domain problem}
\label{formulationsfreqpb}
The frequency domain problem \eqref{eqModifiedHelmholtz} is posed in an unbounded domain. In order to solve it numerically, we need to formulate a problem on a bounded domain, either by using boundary integral formulations or by discretising a finite domain together with an absorbing boundary condition such as Perfectly Matched Layers (PML) \cite{harari2000analytical} or Hardy space infinite elements \cite{Hohage2009}. Both introduce additional poles in the solution operator. A PML layer will lead to an additional continuous spectrum from zero to infinity \cite{Kim:2009us}. Boundary integral formulations have resonances, which are related to the corresponding interior problems. Hence, the convergence results depicted in Section \ref{sec:convergence} depend not only on the scattering poles, but also on the poles introduced by the formulation of the frequency domain problem on a finite domain.

This section gives an overview of possible integral equation formulations for the frequency domain problem and discusses how these formulations introduce additional poles into the solution operator. The frequency domain problem is a modified Helmholtz problem of the form
\begin{align}
\begin{cases}
\displaystyle  \omega^2 U(x)-\Delta U(x) = 0,  & x \in  \Omega^{+}, \\
U(x) = G(x), & x \in \Gamma,
\end{cases}
\label{eqModifiedHelmholtz2}
\end{align}
with $\omega\in\mathbb{C}$ and outgoing boundary conditions towards infinity as described in Section \ref{sec:cq_formulation}.
The Green's function associated with the modified Helmholtz problem is $g_{\omega}(x,y):=\frac{e^{-\omega|x-y|}}{4 \pi |x-y|}$.

We define the single and double layer potential operators for the modified Helmholtz equation as
$$
\left[\mathcal{S_{\omega}}\phi\right](x)=\int_{\Gamma}g_{\omega}(x,y)\phi(y)ds(y),\quad
\left[\mathcal{K_{\omega}}\phi\right](x)=\int_{\Gamma}\frac{\partial}{\partial n(y)}g_{\omega}(x,y)\phi(y)ds(y),\quad x\in\Omega^{+}.
$$
Both operators satisfy the modified Helmholtz equation in the exterior of the domain. We also need the single layer boundary operator $S_{\omega}$ and the double layer boundary operator $K_{\omega}$ defined by
$$
\left[{S_{\omega}}\phi\right](x)=\int_{\Gamma}g_{\omega}(x,y)\phi(y)ds(y),\quad
\left[{K_{\omega}}\phi\right](x)=\int_{\Gamma}\frac{\partial}{\partial n(y)}g_{\omega}(x,y)\phi(y)ds(y),\quad x\in\Gamma.
$$
Let $\gamma_0$ be the exterior trace operator. Then $S_{\omega}=\gamma_0\mathcal{S}_{\omega}$ and $\frac{1}2I+K_{\omega}=\gamma_0\mathcal{K}_{\omega}$, where $I$ is the identity operator. Details of mapping properties for these operators can be found in \cite{COSTABEL:1988vk}.

In this paper we only consider indirect boundary integral formulations. The results for direct boundary integral formulations are very similar.

\subsection{Indirect first kind integral formulation}

An integral formulation of the first kind to solve the modified Helmholtz equation \eqref{eqModifiedHelmholtz2} for a given parameter $\omega\in \mathbb{C}$ is given by
\begin{equation}
\left[S_{\omega} \phi \right] (x) = G(x), \quad x \in \Gamma, \label{eqIndFirstKind}
\end{equation}
The solution in the exterior $\Omega^{+}$ is then obtained as $U=\mathcal{S_{\omega}}\circ S_{\omega}^{-1}G$.
This representation holds for all $\omega$ such that $i\omega\neq k_j$ and $i\omega\neq p_j$,
where the $p_j$ are the scattering poles as defined in Theorem \ref{thm:meromorphic}, and
the $k_j$ are
the eigenfrequencies of the interior Dirichlet eigenvalue problem, satisfying
\begin{align}
\begin{cases}
\displaystyle -\Delta v(x) = k_j^2v(x), & x \in \Omega,\nonumber\\
v(x) = 0, & x \in\Gamma,\nonumber
\end{cases}
\end{align}
for some nonzero $v\in H^1(\Omega)$ (see \cite{Colton:1998vu}). The situation is depicted in Figure \ref{figPolesSL} for the case of backward Euler and a unit sphere as domain. The red dots show the Dirichlet eigenvalues closest to the contour given by the time-stepping rule.

\subsection{Indirect second kind integral formulation}
Using an indirect second kind formulation, we obtain the integral equation
\begin{equation}
\left[ \left( \frac{1}{2} I + K_{\omega}\right) \phi \right] (x) = G(x),\quad x \in \Gamma, \label{eqIndSecKind}
\end{equation}
which gives the representation of the exterior solution in $\Omega^{+}$ as $U=\mathcal{K}_{\omega}\left(\frac12 I+K_{\omega}\right)^{-1}G$.

Similar to the case of the indirect first kind formulation, this representation is valid for all $\omega$ such that $i\omega\neq p_j$ and $i\omega \neq \mu_j$, where the $\mu_j$
are the eigenfrequencies of the interior Neumann eigenvalue problem, satisfying
\begin{align}
\begin{cases}
\displaystyle -\Delta v = \mu^2 v & \text{ in } \Omega, \notag\\
\displaystyle \frac{\partial v}{\partial n} = 0 & \text{ on } \Gamma, \notag
\end{cases}
\end{align}
for some $v\in H^1(\Omega)$. We recall that $n$ is the outgoing normal to $\Omega$.
However, since $0$ is always an eigenvalue of the interior Neumann eigenvalue problem, the value $\omega=0$ is always a pole for the representation as indirect second kind integral equation. Suppose that we use backward Euler as time-stepping rule. Then $\gamma(z)=1-z$, and if the Z-transform of the boundary data has a sufficiently large radius of analyticity, it follows that $\lambda_U=1$. Applying Theorem
\ref{thm:uconvergence}, we obtain the simple convergence estimate
\begin{equation}
\left| u_d^{N_f}(t_n; x) - u_d(t_n; x) \right| = \mathcal{O}\left( \left(\lambda+\epsilon\right)^{N_f} \right)\label{eqCvRateSecKf}
\end{equation}
for any $\epsilon>0$ and $\lambda<1$. Figure \ref{figPolesDL} shows the location of the poles with respect to the contour given by the backward Euler rule for the case of the indirect second kind formulation. The pole at zero is always closest to the contour and dominates the convergence behaviour.
\begin{figure}[h]
\centering \subfigure[First Kind Integral Formulation.]{\label{figPolesSL}\includegraphics[width=0.45\textwidth]{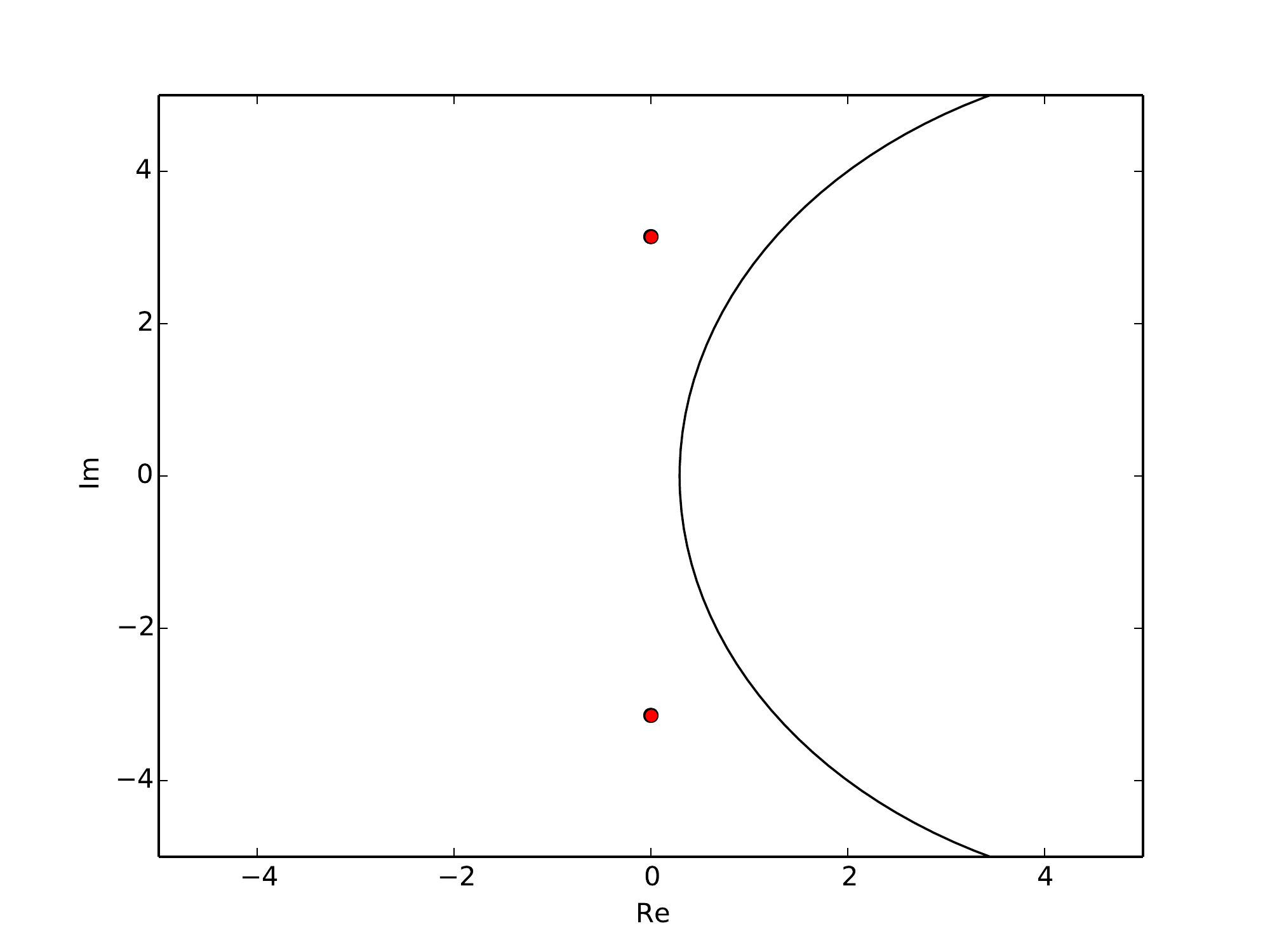}} \quad \subfigure[Second Kind Integral Formulation.]{\label{figPolesDL}\includegraphics[width=0.45\textwidth]{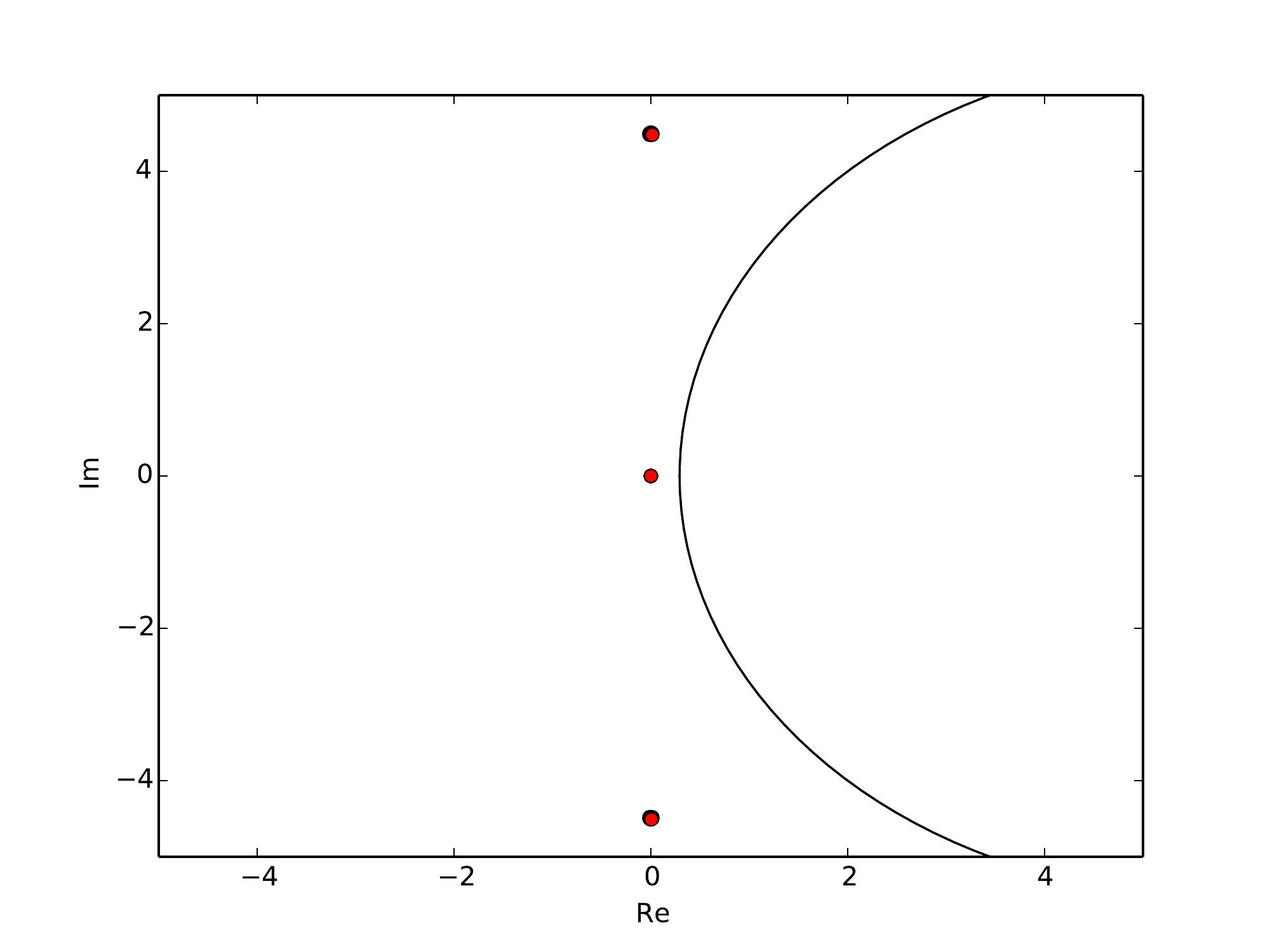}}
\caption{Poles located near the contour for the indirect first and second kind integral formulations.}
\label{figPolesContour1}
\end{figure}

\subsection{Indirect combined integral formulation}
The indirect formulation of the second kind always has a pole at zero while the indirect formulation of the first kind has a pole related to the first eigenvalue of the interior Dirichlet problem. It is therefore sensible to consider a combined formulation to try to push away the smallest magnitude pole introduced by the boundary integral formulation. A combined formulation to solve \eqref{eqModifiedHelmholtz2} takes the form
\begin{equation}
\label{eqIndCombined}
\left[ \frac{1}{2} I + K_{\omega} +\eta S_{\omega}\right]\phi (x) = G(x),\quad x \in \Gamma.
\end{equation}
The representation of the solution in $\Omega^{+}$ is therefore given by
$$
U=\left[\mathcal{K}_{\omega}+\eta\mathcal{S}_{\omega}\right]\left(\frac{1}{2} I + K_{\omega} +\eta S_{\omega}\right)^{-1}G.
$$
The following result is a reformulation of \cite[Theorem 3.33]{Colton:1983uk} for the modified Helmholtz equation \eqref{eqModifiedHelmholtz2}.
\begin{theorem}
\label{thm:combinedFormulationSolvability}
Let $\eta\neq 0$ with $\mathrm{Re}\{\eta\}=0$ and $\mathrm{Im}\{\eta\}\mathrm{Im}\{\omega\}\geq 0$. Then the combined formulation \eqref{eqIndCombined} is uniquely solvable for all frequencies $\omega$ satisfying $\mathrm{Re}\{\omega\}\geq 0$.
\end{theorem}

In addition to singularities at the scattering poles, the combined formulation has resonances at the eigenfrequencies $\nu$ of the modified interior impedance eigenvalue problem
\begin{align}
\begin{cases}
-\Delta v = -\nu^2 v &  \text{in } \Omega, \\
\frac{\partial v}{\partial n}+\eta v = 0& \text{on } \Gamma
\end{cases}
\label{eqInteriorImpedance}
\end{align}
for some $v\in H^1(\Omega)$. For real $\eta$ it can be readily seen that all eigenfrequencies $\nu$ lie on the imaginary axis. Moreover, as $\eta\rightarrow 0$ the smallest eigenfrequency $\nu$ approaches $0$ since $\eta=0$ corresponds to the Neumann case. For $\eta\rightarrow\infty$ the smallest eigenfrequency $\nu$ approaches the smallest eigenfrequency of the Dirichlet case.

If $\mathrm{Im}\{\eta\}>0$, then by Theorem \ref{thm:combinedFormulationSolvability} and the fact that if $\nu$ is an eigenfrequency, then also $-\nu$ is an eigenfrequency, it follows that the interior impedance eigenvalues can only be located in the lower right quadrant and in the upper left quadrant of the complex plane. Hence, singularities can occur close to or in the interior of the contour defined by the values $\frac{\gamma(z)}{c\Delta t}$, $|z|=\lambda$. This is demonstrated in Figure \ref{figPolesCLc} for $\eta=i$. We now have a pole inside the contour given by the backward Euler rule, and we have to modify the contour (e.g. by choosing $\lambda<1$) to remedy the situation.

To avoid this problem, one strategy is to choose $\eta = \omega$ in \eqref{eqIndCombined}. The mapping properties of the resulting combined field operator were analysed in \cite{monk:2014}. There it is shown that the combined potential operator has a bounded $L^2$ inverse for all wavenumbers satisfying $\mathrm{Re}\{\omega\}>0$. The corresponding situation is depicted in Figure \ref{figPolesCLk}. The location of poles in this combined formulation is not symmetric any more. However, as in the case of the second kind integral formulation, we still have a pole at zero. Hence, for the asymptotic rate of convergence of the CQ approximation for the backward Euler rule there is no difference between the second kind formulation and the combined formulation with $\eta=\omega$.


\begin{figure}[h]
\centering \subfigure[Combined Integral Formulation with $\eta =  i$]{ \label{figPolesCLc} \includegraphics[width=0.45\textwidth]{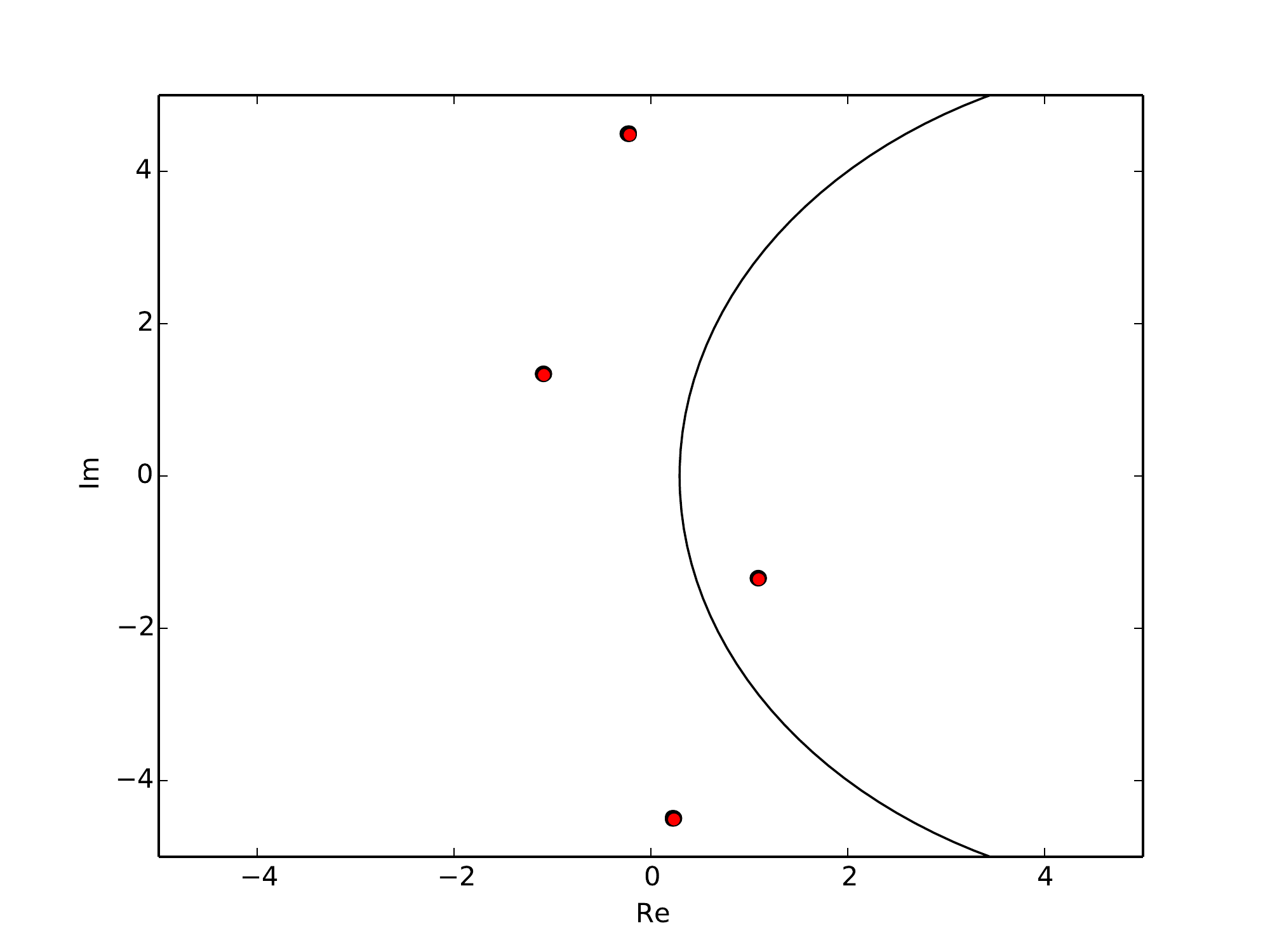}} \quad \subfigure[Combined Integral Formulation with $\eta=\omega$]{\label{figPolesCLk}\includegraphics[width=0.45\textwidth]{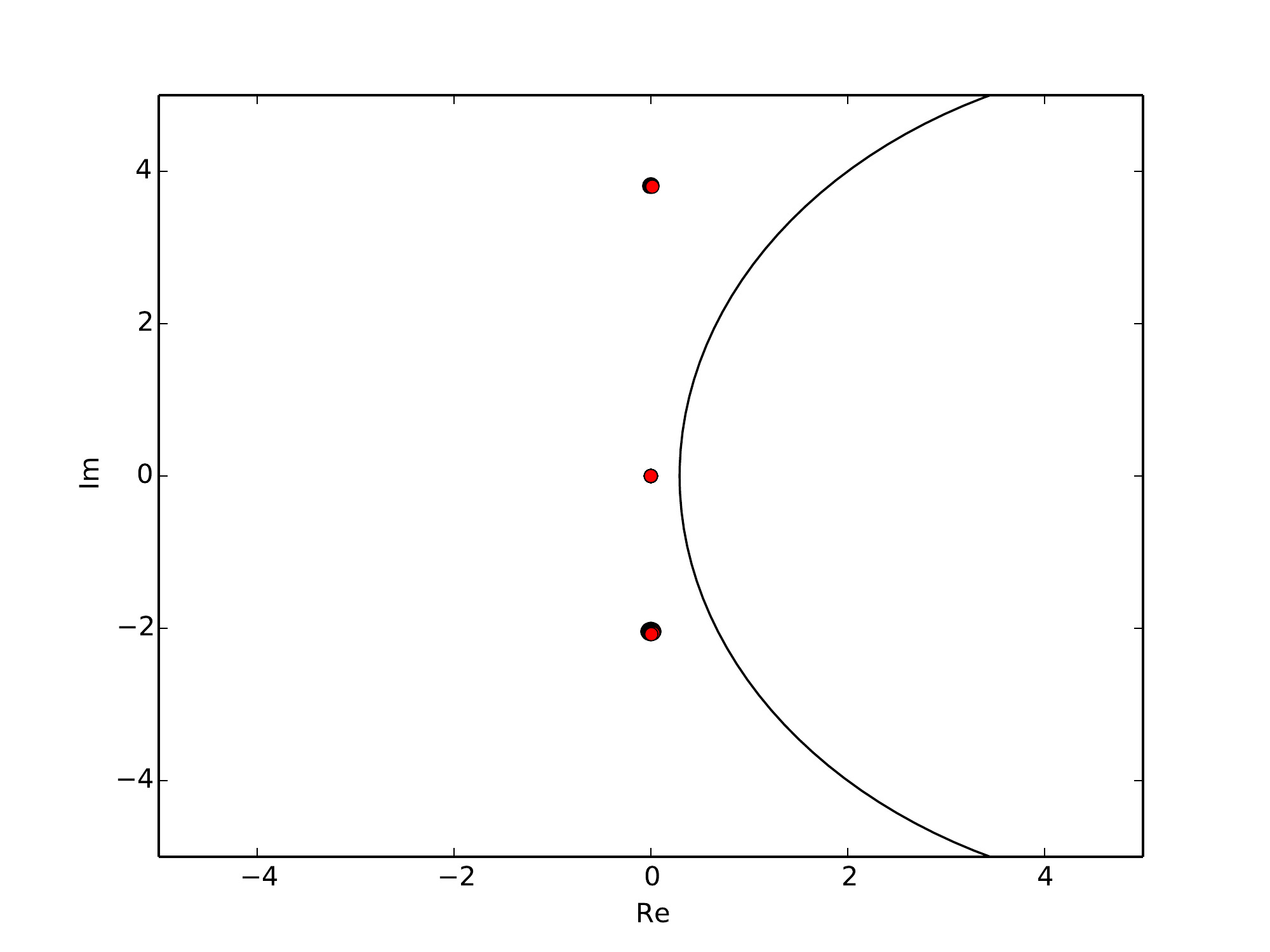}}
\caption{Poles located near the contour for two different combined integral formulations for the scattering by the unit sphere.}
\label{figPolesContour}
\end{figure}

\section{Numerical results} \label{secNumRes}
In this section we demonstrate the numerical behaviour of the CQ method as the number of frequencies $N_f$ is increased for fixed time $N_t$. The model problem is the acoustic wave equation
\begin{align}
\begin{cases}
\displaystyle \frac{\partial^2 u}{\partial t^2}(t; x) - c^2 \Delta_xu(t; x) =0, & x \in \Omega^{+}, \\
\displaystyle u(0; x) = \frac{\partial u}{\partial t}(0; x) = 0, & x \in\Omega^{+}, \\
u(t; x) = g(t; x), & x \in \Gamma,
\end{cases} \label{eqPbNumeric}
\end{align}
with boundary condition
\begin{align}
g(t;x)=-\cos\left(2\pi\left(t-\frac{d\cdot x}{c}\right)f\right)e^{-\frac{\left(t-t_p-\frac{d\cdot x}{c}\right)^2}{2\sigma^2}}, \label{eqdefgincident}
\end{align}
Here, we use the parameters $t_p= 10^{-3}$, $\sigma = \frac{6}{2000 \pi}$, $c = 343$, $d = \left(1,0,0\right)^t$. The final time is $T_f = 20\cdot 10^{-3}$ and the number of time steps is $N_t = 40$.
We recall that $g(t_n;x)$ decreases exponentially as $n \to \infty$. Hence, for the radius of analyticity $\lambda_G$ of the Z-transform of the boundary data we obtain $\lambda_G = \infty$ and according to Theorem \ref{thm:u_analytic}, $\lambda_U = \lambda_{\mathcal{B}}$  with $\lambda_{\mathcal{B}}$ defined by (\ref{eqDefLambdaB}). In the first part of this section $\Omega$ will be the unit sphere in $\mathbb{R}^3$. Later, we will present results for a more challenging trapping domain.

We evaluate the time-domain solution in the observation domain $$\Omega_{\text{obs}} = \left\{ x \in [-3,3], y\in [-3,3] , z=0; \sqrt{x^2+y^2} > 1\right\}.$$ The maximum pointwise error in $\Omega_{\text{obs}}$ is
measured as
\begin{align}
\text{AbsDiff}(N_f) = \max_{n \in [0,N_t]} \left\Vert u_d^{N_f}(t_n;x) - u_{\text{ref}}(t_n;x) \right\Vert_{L^{\infty}(\Omega_{\text{obs}})}.\label{eqDefAbsErr}
\end{align}
Reference solutions are computed by using a very high number of frequencies $N_f$ in the Laplace domain. All numerical results in this section were computed using the boundary element package BEM++ (\url{www.bempp.org}) \cite{SmigajBetcke13}.

\subsection{Validation of the theoretical rate of convergence}
We first compare the predicted rate of convergence in Theorem \eqref{thm:uconvergence} with the observed convergence in the case of the indirect second kind integral formulation (\ref{eqIndSecKind}) and backward Euler time-stepping rule.
The location of the poles for this formulation was depicted in Figure \ref{figPolesDL}. The pole at zero dominates the convergence. Comparisons of the theoretical estimated rate of convergence and the measured decay of $\text{AbsDiff}(N_f)$ for various $\lambda$ are shown in Figure \ref{figSecKind}. There is a very close match between the theoretical estimate and the achieved rate of convergence. It is interesting to consider the point $N_f=N_t$, where we have the same number of frequency domain solves as there are time-steps. This corresponds to previously proposed CQ methods. As expected, the error becomes smaller at this point as $\lambda$ decreases. However, for very small $\lambda$ the convergence soon starts to level off due to numerical instabilities with small $\lambda$. This is further shown in Figure \ref{figChoiceLambda}, where the maximum achievable accuracy in dependence of $\lambda$ is demonstrated for the second kind formulation.
\begin{figure}
\centering
\subfigure[Theoretical and numerical results for the scattering by the unit sphere using the indirect second kind integral formulation.] {\label{figSecKind} \centering \includegraphics[height=0.3\textheight]{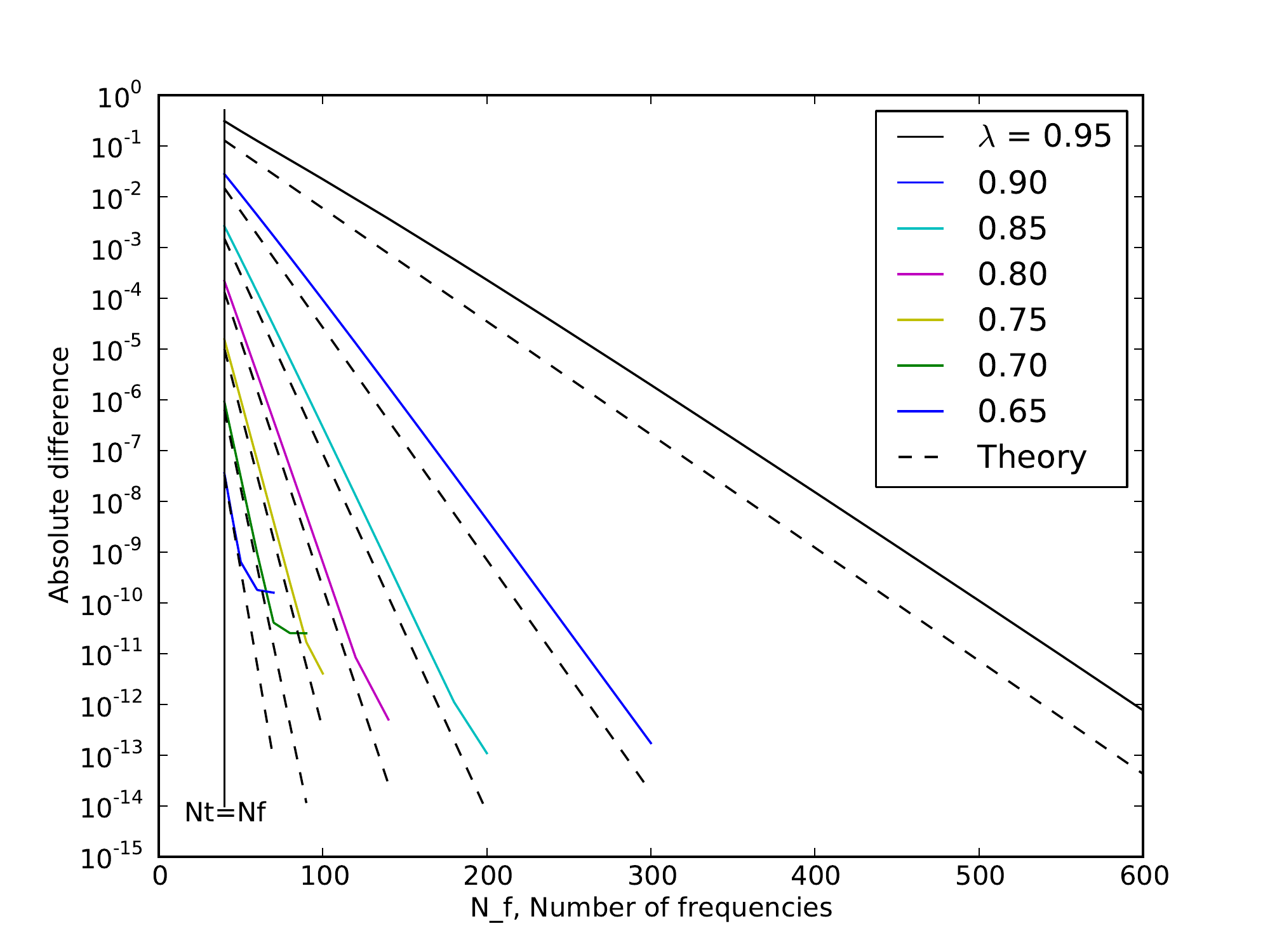}}\\
\subfigure[Limit of achievable accuracy when $\lambda$ is small for the unit sphere using the indirect second kind integral formulation.]{\label{figChoiceLambda}  \centering \includegraphics[height=0.3\textheight]{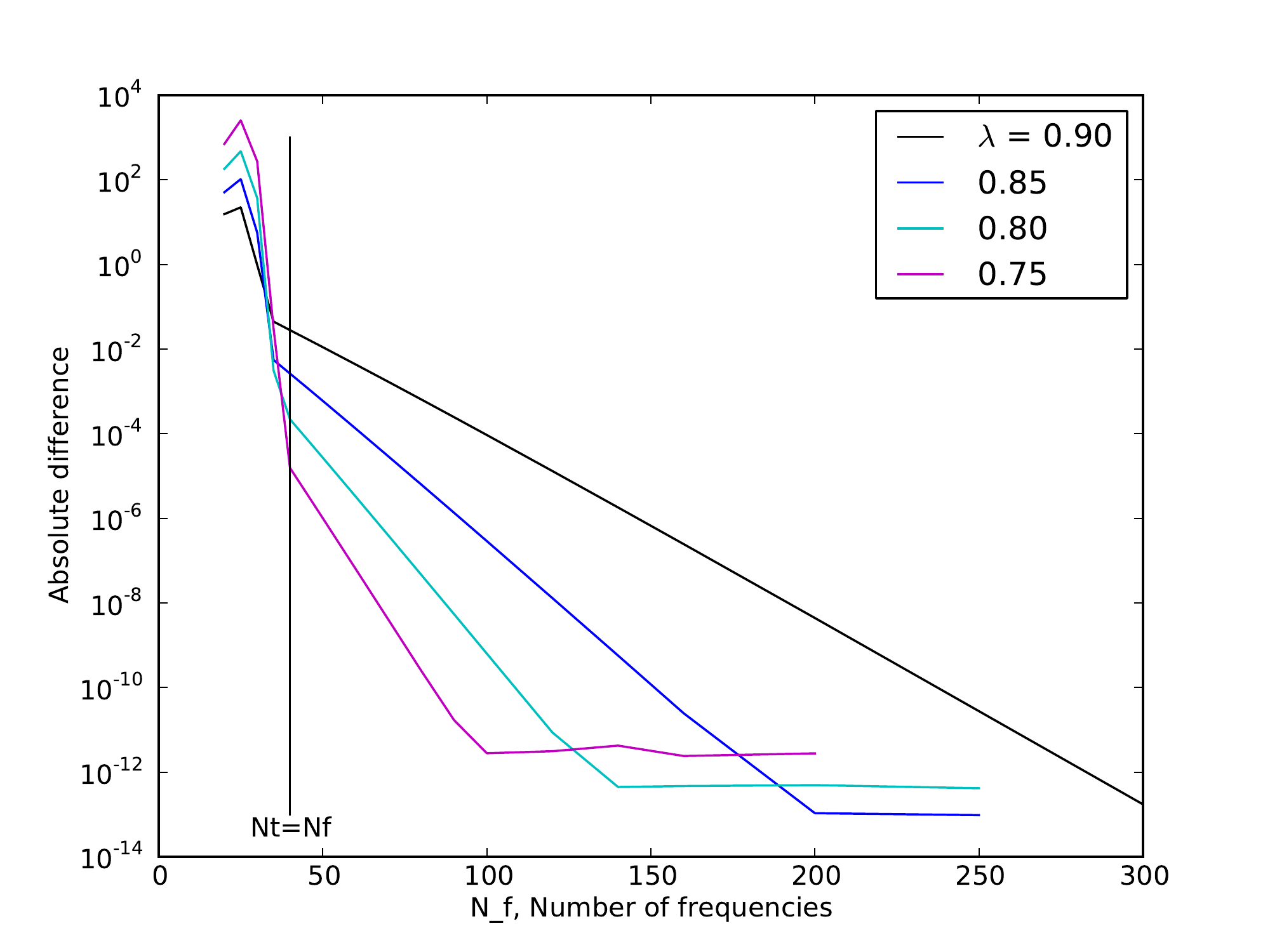}}
\caption{Convergence on the unit sphere for the indirect second kind integral equation.}
\end{figure}

We now demonstrate how the rate of convergence changes for different coupling coefficients in the combined integral formulation \eqref{eqIndCombined}.
We fix $\lambda=0.95$. Figure \ref{figCombinedEta1} shows the convergence for constant real $\eta = 1$. In this case we have
$\lambda_U \approx  1.0346$ and therefore a convergence rate of $\left(\frac{\lambda}{\lambda_U}\right)^{N_f}\approx 0.9182^{N_f}$, whereas for the standard second kind formulation we would only expect a rate of convergence of $0.95^{N_f}$.
The combined formulation with $\eta = \omega$ converges with a rate of $0.95^{N_f}$, the same rate as the second kind formulation, as shown in Figure \ref{figCombinedEtaW}. However, comparing Figure \ref{figCombinedEtaW} and \ref{figSecKind} it becomes obvious that the combined formulation with $\eta = \omega$ is significantly more accurate than the second kind formulation for the same number of frequencies. Indeed, at the point $N_f = N_t$ we have an error of $3.106\cdot 10^{-1}$ for the second kind formulation and an error of $8.572\cdot 10^{-3}$ for the combined formulation. Hence, in practice the combined formulation may be preferable.
\begin{figure}[h]
\subfigure[Absolute difference for the unit sphere using the indirect combined integral formulation with real combining coefficient $\eta=1$.]{\label{figCombinedEta1}  \centering \includegraphics[height=0.26\textheight]{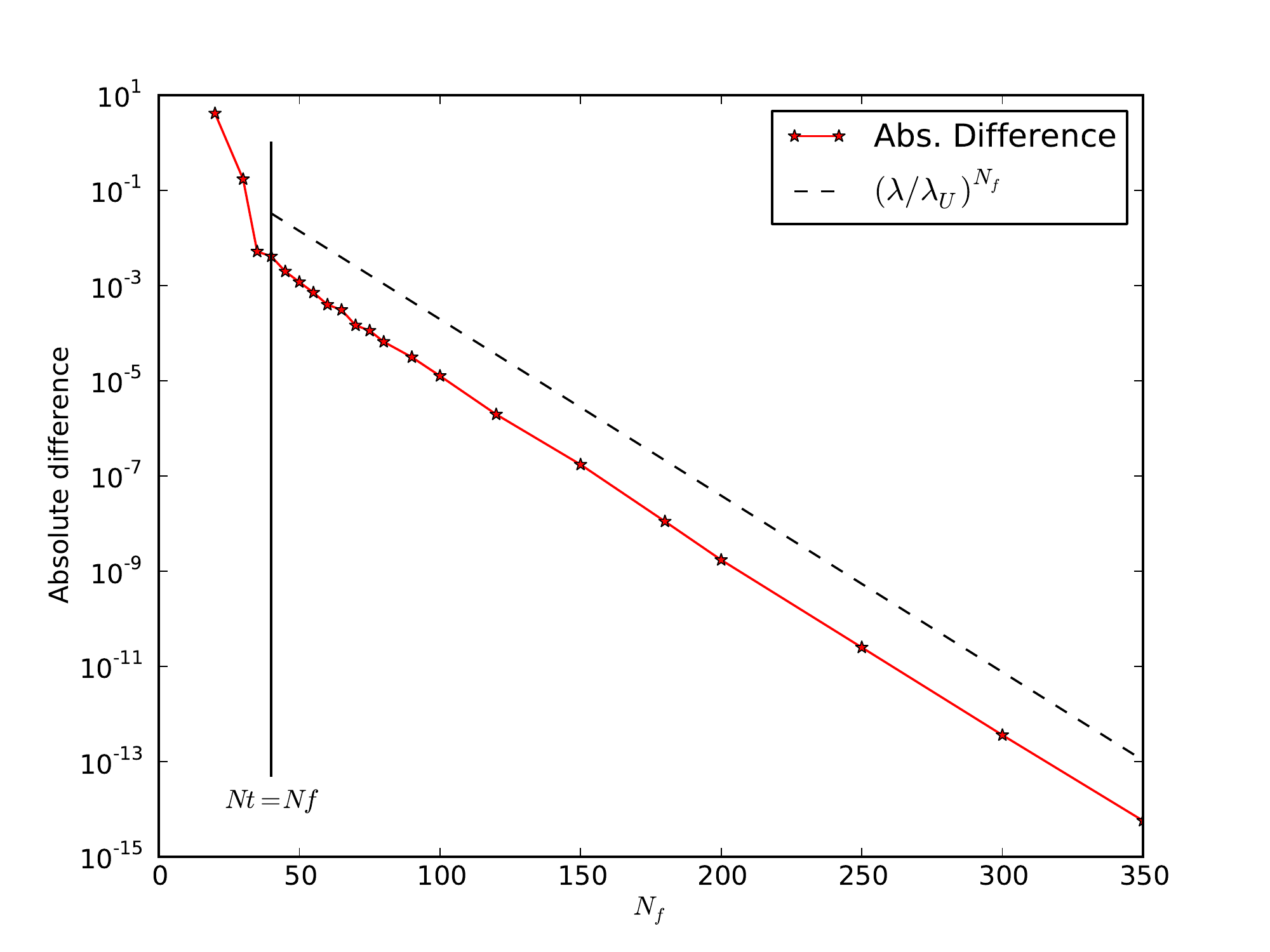}} \quad \subfigure[Absolute difference for the unit sphere using the indirect combined integral formulation with varying coefficient $\eta=\omega$.] {\label{figCombinedEtaW}  \centering \includegraphics[height=0.26\textheight]{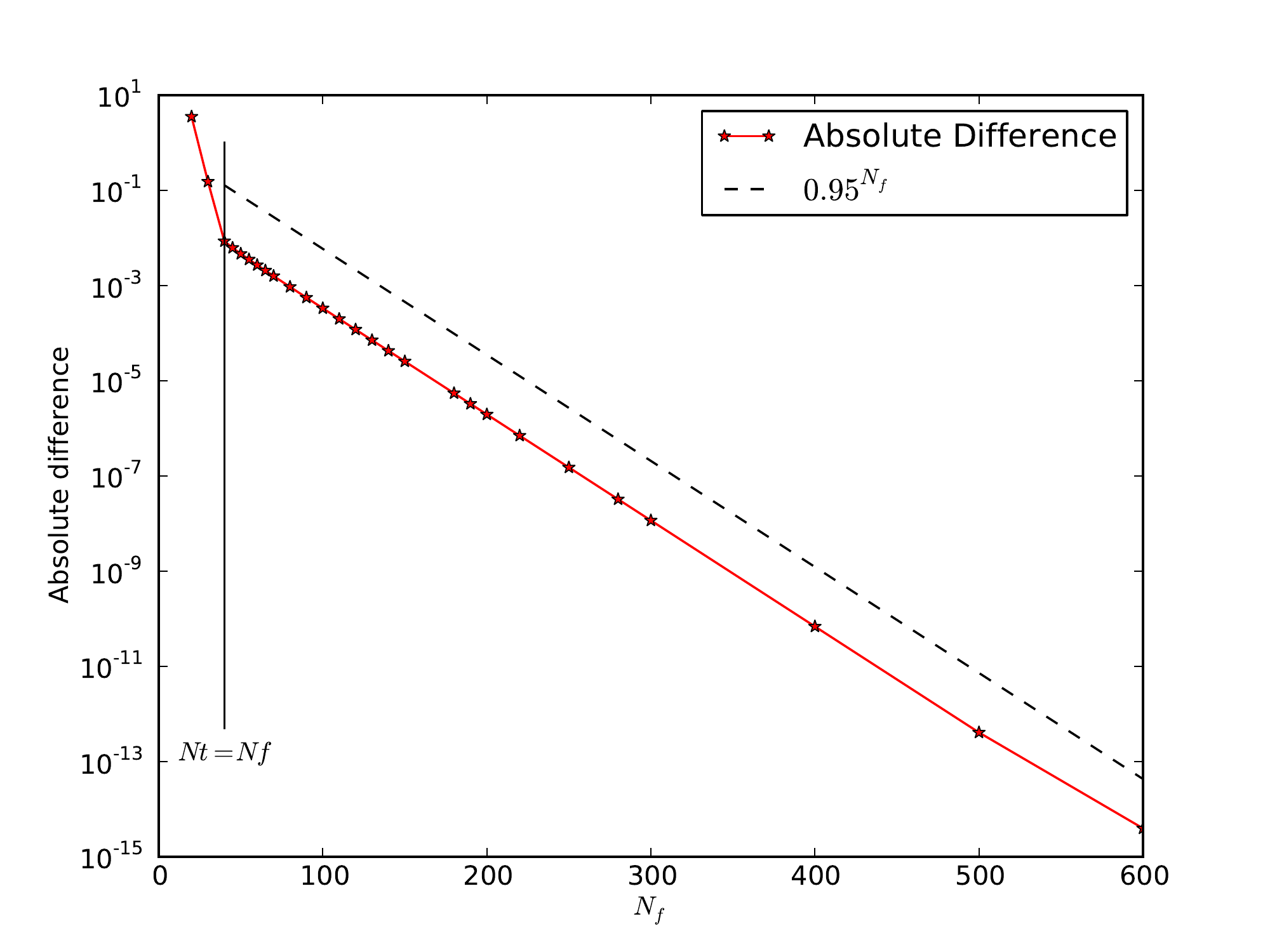}}
\caption{Two indirect combined integral formulations.}
\end{figure}

When $\eta=i$ and $\lambda=0.95$, there is a pole inside the contour and so Theorem \ref{thm:uconvergence} is no longer usable, but the solution still seems to converge when $N_f \to \infty$, as demonstrated in Figure \ref{figAbsDiff_etai}. However, the rate of convergence does not seem to be exponential.

\begin{figure}[h]
\centering \includegraphics[height=0.26\textheight]{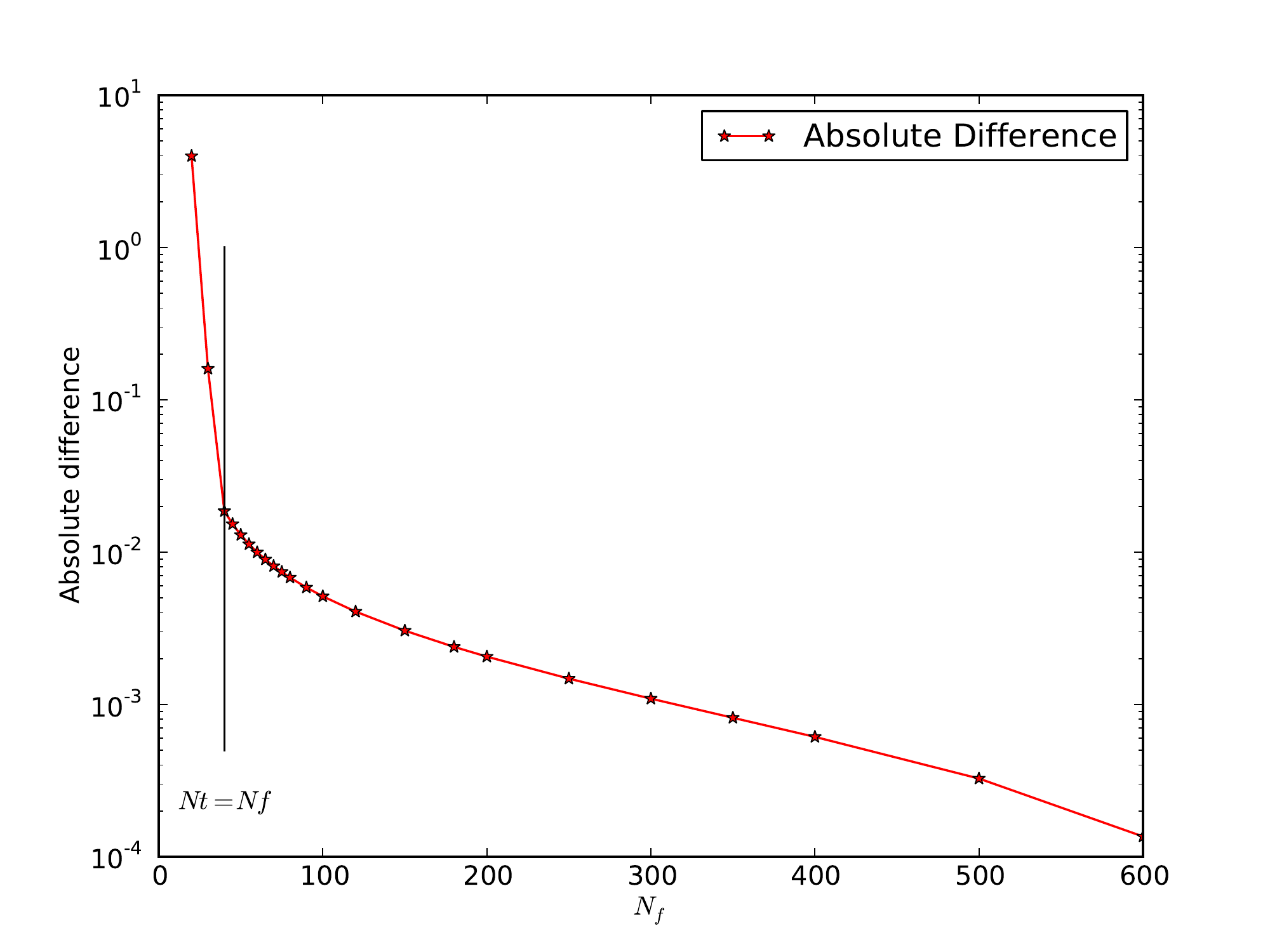}
\caption{Convergence of the numerical solution when a pole is inside the contour (case $\eta = i$ and $\lambda=0.95$).}
\label{figAbsDiff_etai}
\end{figure}

\subsection{Comparison of the rate of convergence for backward Euler and BDF-$2$}
It is interesting to compare backward Euler with BDF-2 as $N_f\rightarrow\infty$.  Figure \ref{figBDF12contour} depicts the contour for backward Euler and BDF-2. We observe that the pole is closer to the BDF-$2$ contour than to the backward Euler contour. Hence, the rate of convergence of the convolution quadrature approximations to the exact time-stepping values will be slower for BDF-2 than for backward Euler, unless the pole is at zero, in which case both rates of convergence are identical.

Figure \ref{figBDFcompare} confirms this by presenting the measured and the theoretical rate of convergence for these two schemes using a combined integral formulation with $\eta = 20$ and $\lambda = 0.9$. In the case of backward Euler we have $\lambda_U\approx 1.1318$ and for BDF-$2$ $\lambda_U\approx 1.0118$. We note, however, that BDF-$2$ is still a significantly more accurate scheme for the solution of the underlying wave equation, as it is second order in time, while backward Euler is only first order accurate.
\begin{figure}[h]
\subfigure[Contour used for backward Euler (black) and BDF-$2$ ({\color{blue}{blue}}).]{\label{figBDF12contour}
\centering \begin{tikzpicture}[scale=0.24]
\def\coef{5.83}
\draw[domain=0.0:6.283185, smooth,samples=140,variable=\x, black] plot ({(1-0.9*cos(deg(\x)))*\coef},{(-0.9*sin(deg(\x)))*\coef});
\draw[domain=0.0:6.283185, smooth,samples=230,variable=\x, blue] plot ({\coef*(1.5+0.5*(0.9*0.9*cos(deg(2.0*\x))*cos(deg(2.0*\x)))-0.5*(0.9*0.9*sin(deg(2.0*\x))*sin(deg(2.0*\x)))-2.0*0.9*cos(deg(2.0*\x)))},{\coef*(0.9*0.9*cos(deg(2.0*\x))*sin(deg(2.0*\x))-2*0.9*sin(deg(2.0*\x)))});
\draw[thick, ->] (-1,0) -- (25,0);
\draw[thick, ->] (0,-12) -- (0,12);
\draw (\coef,0) -- (0,3.0908);
\draw[blue] ({\coef*1.5},0) -- (0,3.0908);
\fill[red] (0,3.0908) circle (8pt);
\fill[black] (\coef,0) circle (7pt);
\fill[blue] ({\coef*1.5},0) circle (7pt);
\end{tikzpicture}} \quad
\subfigure[Absolute difference of the backward Euler and BDF-$2$ schemes.]{
\label{figBDFcompare}  \centering \includegraphics[height=0.26\textheight]{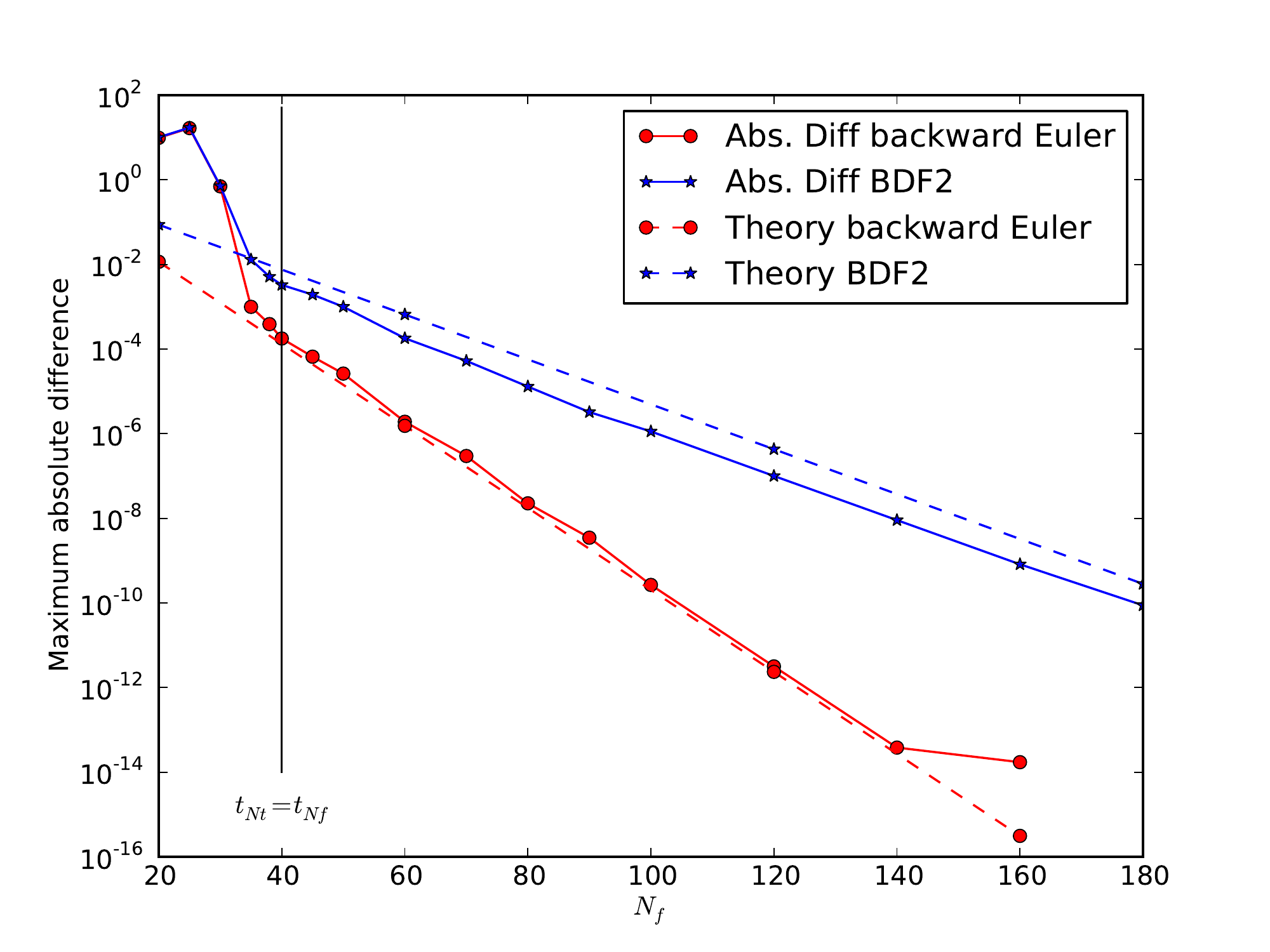}}
\caption{Comparisons between the backward Euler and BDF-$2$ schemes.}
\end{figure}

\subsection{Trapping domain}
Until now, we were studying  the solution of problem (\ref{eqPbNumeric}) when $\Omega$ is the unit sphere. We now consider the elliptic cavity shown in Figure \ref{figMeshCavity}. It is a three dimensional version of the elliptic cavity studied in \cite{Betcke:2010tt}.
For the two dimensional case it was shown in that paper that there exists a sequence of wavenumbers along the real axis for which the norm of the combined potential operator for the Helmholtz equation grows exponentially.

For the three dimensional elliptic cavity it is not possible to evaluate explicitly the poles of the solution operator. Denote by $\mathbb{A}(\omega)$ the matrix obtained from a Galerkin discretisation of the combined potential operator
$\left[ \frac{1}{2} I + K_{\omega} +\eta S_{\omega}\right] $ on the boundary $\Gamma$ of the trapping domain. Let $\mathbb{M}$ be the associated mass matrix and $\mathbb{M}=CC^H$ its Cholesky decomposition.
Then a simple way to have an idea of the location of the poles is to plot
\begin{align}
p(\omega) = \left\Vert  \mathbb{A}^{-1}(\omega) \right\Vert_{L^2(\Gamma)} = \left\Vert C^{-1} \mathbb{A}(\omega)^{-1} C^{-H} \right\Vert_2. \notag
\end{align}
If $z$ is a pole, then $p(\omega) \to \infty$ when $\omega \to z$.
We used $\eta = 1$ in order to have poles on the imaginary axis. Figure \ref{figPoleElliptic} shows $p(\omega)$ for $\omega \in \left[0, 4i\right]$ and allows to find the closest pole $z_1 \approx 1.7718i$, giving an estimated rate of convergence of $0.90896^{N_f}$ for backward Euler. The observed rate of convergence in \ref{figMeshCavityCV} matches very closely this predicted rate.
Figure  \ref{figTimeDomainSol} provides snapshots of the corresponding time-domain solution at four different time steps.

\begin{figure}
\centering\label{figMeshCavity}\includegraphics[width=0.49\textwidth]{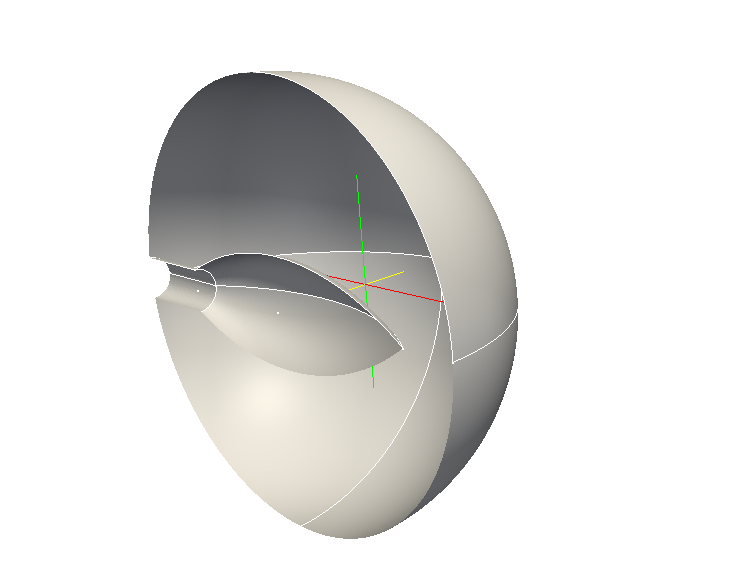}
\caption{Geometry of the trapping domain.}
\end{figure}

\begin{figure}
 \subfigure[$L^2$-norm of the inverse of the combined potential along the imaginary axis.]{\label{figPoleElliptic} \includegraphics[width=0.49\textwidth]{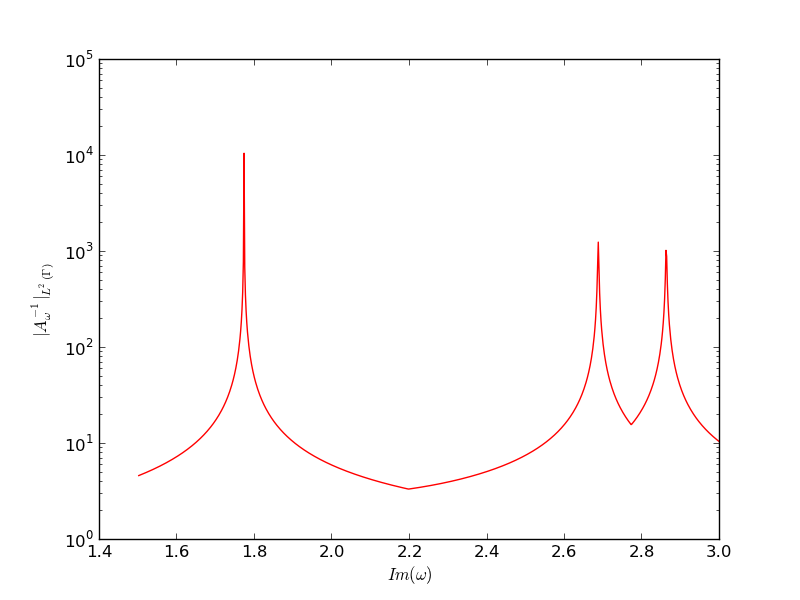}} \subfigure[Convergence of the solution for the elliptic cavity.]{\label{figMeshCavityCV}\includegraphics[width=0.49\textwidth]{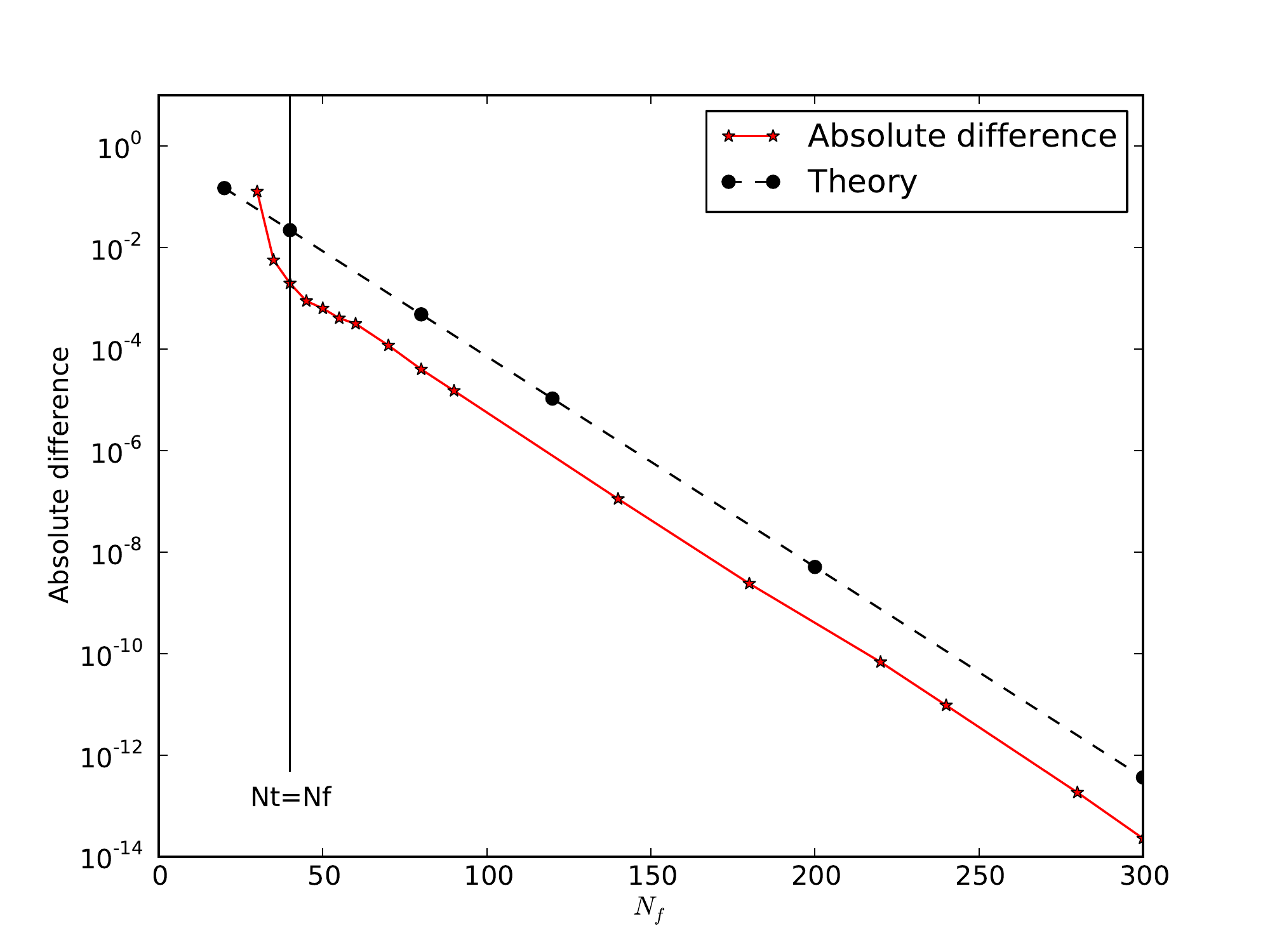} }
\caption{Trapping domain: Location of the nearest pole and absolute difference using an indirect combined formulation with $\eta=1$ and backward Euler. The closest pole is located at $1.7718i$ and therefore $\lambda_U \approx 1.045$, giving a predicted rate of convergence of $(\lambda/\lambda_U)^{N_f} \approx 0.90896^{N_f}$.}
\label{figTrapping}
\end{figure}
\begin{figure}[h]
\subfigure[Time-domain solution at time $0.0045$s.]{\label{figTimeDomain1}\includegraphics[width=0.49\textwidth]{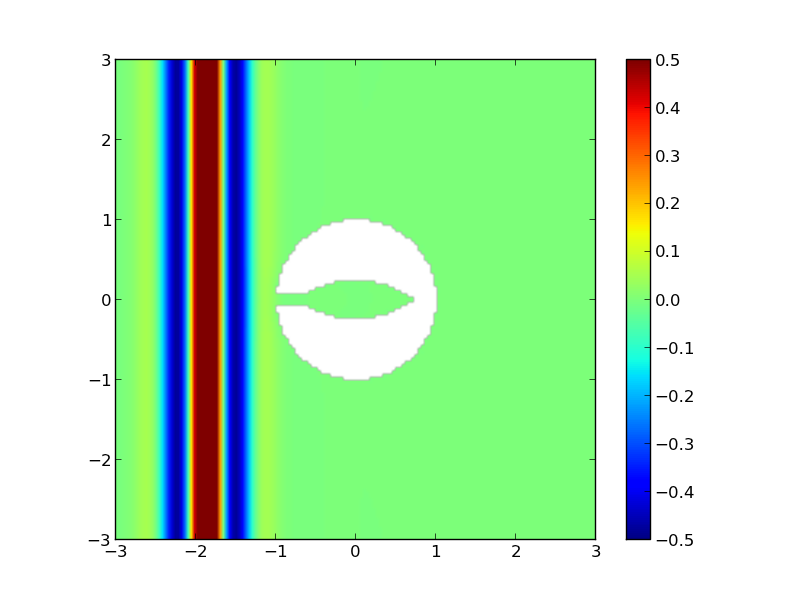}} \, \subfigure[Time-domain solution at time $0.0075$s.]{\label{figTimeDomain2} \includegraphics[width=0.49\textwidth]{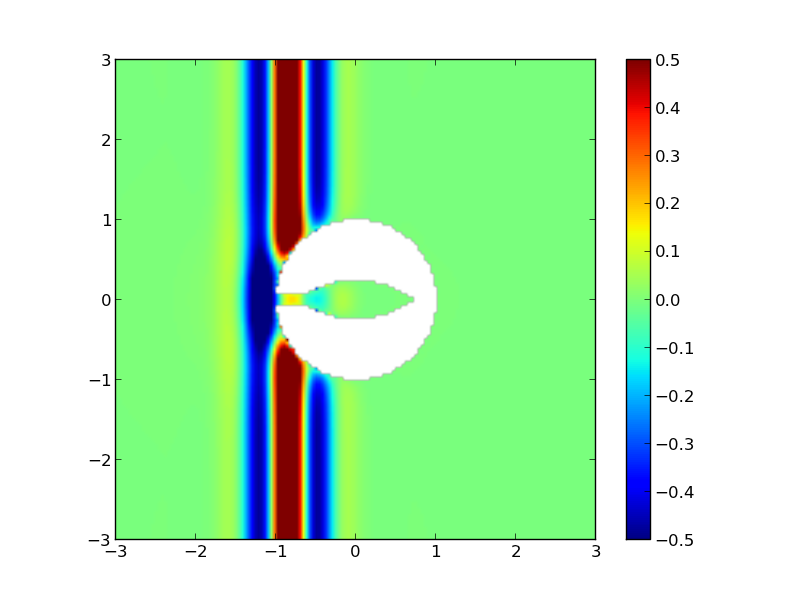}}\\[0.5pt]
\subfigure[Time-domain solution at time $0.01$s.]{\label{figTimeDomain3}\includegraphics[width=0.49\textwidth]{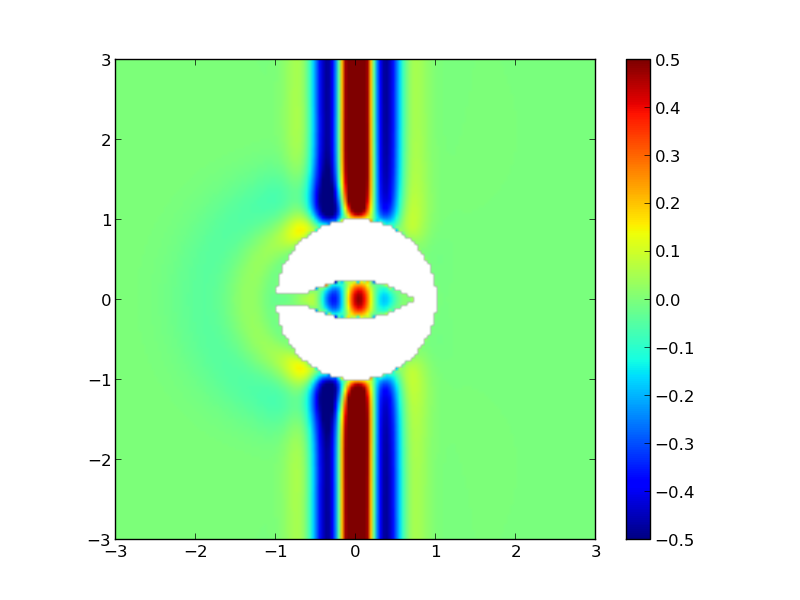}} \, \subfigure[Time-domain solution at time $0.0125$s.]{\label{figTimeDomain4} \includegraphics[width=0.49\textwidth]{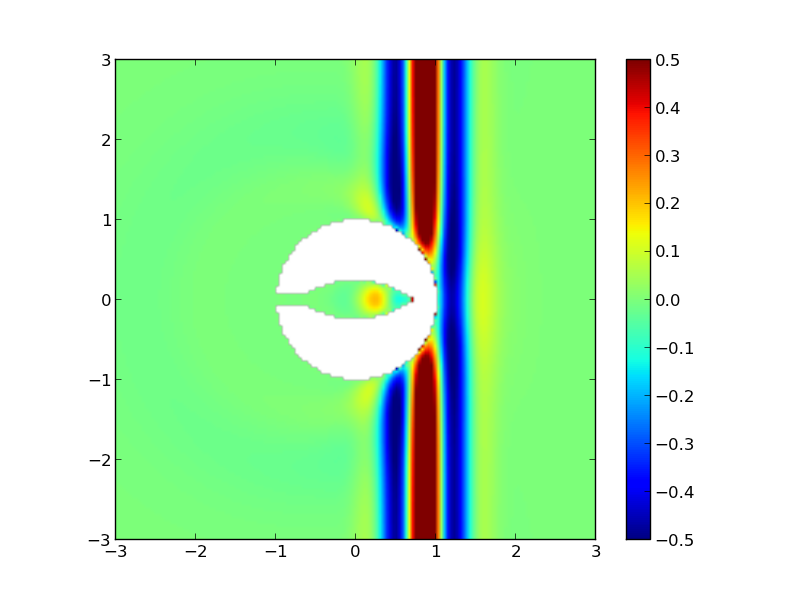}}\\[0.5pt]
\caption{Solution of the scattering by the sphere with the elliptic cavity with backward Euler (Figure \ref{figMeshCavity}) in the plane $z=0$ for different time steps with $\eta=1$ using $N_f = 300, N_t = 40, T_f = 20\cdot 10^{-3}$s and $\lambda=0.95$. }
\label{figTimeDomainSol}
\end{figure}

\subsection{Convergence of Runge-Kutta methods}
\label{sec:rkresults}
We solve problem \eqref{eqPbNumeric} with incident wave \eqref{eqdefgincident}, but this time using the Radau IIa Runge-Kutta method. We use a combined integral formulation with $\eta=\omega$. Hence, for the scalar problem zero is the closest pole. The interesting question is how the vector Helmholtz problem
\eqref{eqVectormodHelm} underlying the Radau IIa formulation influences the rate of convergence, and in particular whether the singularity of the eigenvalue decomposition of $\Delta(z)$ at $z=3\sqrt{3}-5$ is reflected in the observed convergence rate.  Figures \ref{figRKcfie090} and \ref{figRKcfie095} show the absolute difference between the numerical solutions obtained for different numbers of frequencies and a reference result computed with a large number of frequencies, respectively for $\lambda = 0.90$ and $\lambda = 0.95$. It is interesting to observe that the convergence consists of two phases: an initial phase with a significantly faster rate of convergence and then an asymptotic (at least to machine precision) behaviour that shows the same rate of convergence as we would expect for the corresponding scalar solution operator.
Hence, the singularity of the eigenvalue decomposition of $\Delta(z)$ at $z=3\sqrt{3}-5$ does not seem to influence the convergence behavior.
At a much smaller scale, the initial superconvergence behavior can also be observed for the multistep case in Figure \ref{figAbsDiff_etai}. Our current asymptotic analysis does not explain these transient phenomena.

\begin{figure}[h]
 \subfigure[Convergence for $\lambda = 0.90$.]{\label{figRKcfie090} \includegraphics[width=0.49\textwidth]{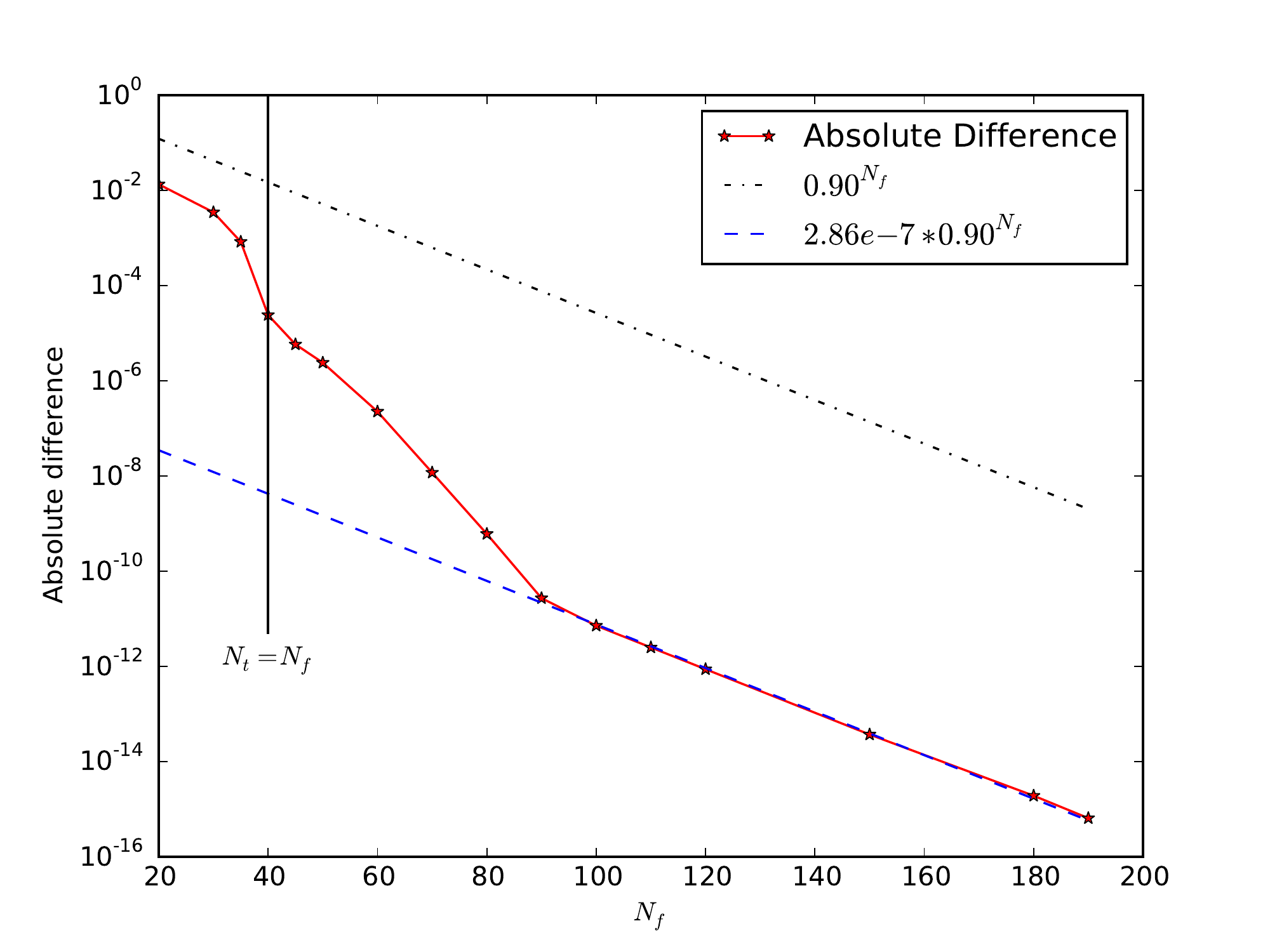}} \subfigure[Convergence for $\lambda = 0.95$.]{\label{figRKcfie095}\includegraphics[width=0.49\textwidth]{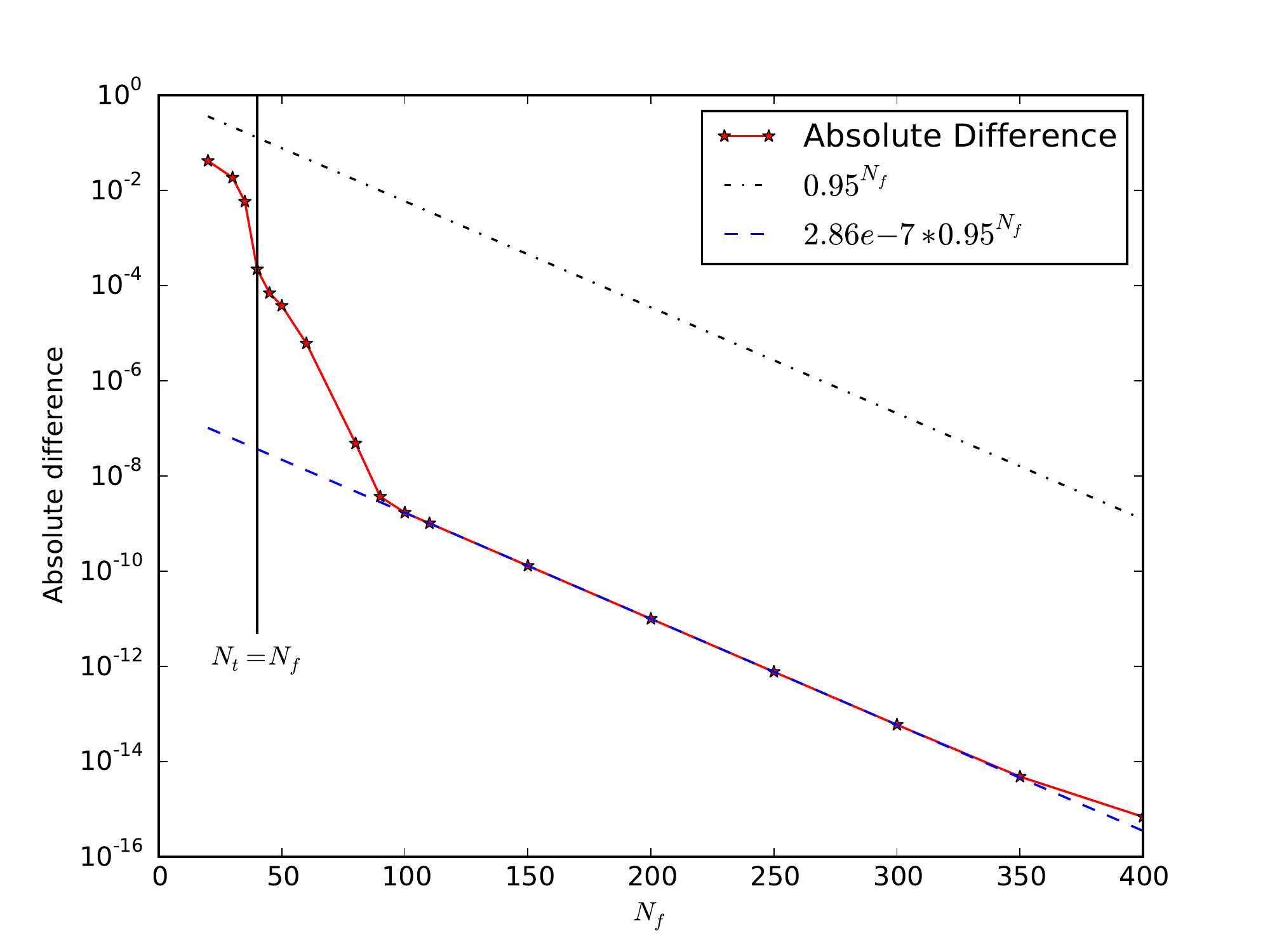} }
\caption{Convergence of the solution obtained with a $2$-stages Radau IIa Runge-Kutta scheme for the Combined integral formulation with $\eta = \omega$ for the scattering by the unit sphere for two different radii of the contour.}
\label{figRKconvergence}
\end{figure}

\subsection{Stability of the solution}
While previously we considered the rate of convergence of the convolution quadrature approximation $u_d^{N_f}$ to the exact time stepping solution $u_d$ for various boundary integral formulations, we want to conclude the numerical examples with a comparison of $u_d^{N_f}$ to the exact solution $u$ for different boundary integral formulations.

We use a boundary condition of the form
\begin{align}
f(t) = b (a t)^m e^{-p t}, \label{eq:condition_analytical}
\end{align}
to get the exact radiating solution
\begin{align}
u_e(r,t) =  b H\left(t + \frac{1-r}{c}\right) \left(a \left(t + \frac{1-r}{c}\right) \right)^m e^{-\frac{p \left(t + \frac{1-r}{c}\right)}{r}},
\end{align}
where $H$ denotes the Heaviside function. We use the $2$-stages Runge-Kutta Radau IIa scheme to discretize in time and $4$ different integral formulations:
\begin{itemize}
\item an indirect first kind integral formulation, see \eqref{eqIndFirstKind}, denoted $SL$,
\item a second kind integral formulation, see \eqref{eqIndSecKind}, denoted $DL$,
\item an indirect combined integral formulation, see \eqref{eqIndCombined}, with $\eta = 1$,
\item an indirect combined integral formulation, see \eqref{eqIndCombined}, with $\eta = \omega$, the wavenumber.
\end{itemize}

The numerical comparison is performed using the boundary condition \eqref{eq:condition_analytical} with $a = 25, b=300, m = 10$ and $p = 150$ and evaluating the solution at points located on a circle of radius $1.1$. The final time is $T_f = 0.30$. We use $N_t = 80$ time steps and $\lambda = 0.95$. To demonstrate the influence of the number of frequency solves on the number of time steps we performed the computation with $N_f=100$ and $N_f=400$ frequency solves.

In Figure \ref{figCompareIEform} we compare the error of the numerically computed solution to the exact solution for growing time $t$. It is remarkable that the two formulations with a pole at the origin (DL and $\eta = w$) deteriorate quickly while the two solutions with poles away from $0$ have a small relative error throughout the observed time interval. This behavior is independent on wheter we choose $N_f=100$ or $N_f=400$. The bottom plot shows as comparison the absolute error for $N_f=400$. It shows that as the analytical solution converges to zero the DL and $\eta=\omega$ case remain bounded away from zero. 

Figure \ref{figCompareIEform} also nicely demonstrates the influence of the error of the underlying time-stepping scheme. For $N_f=100$ the error of the $\eta=1$ formulation is larger than that of the SL formulation. However, for $N_f=400$ both errors are identical and indeed there is no difference in error for the SL formulation between $N_f=100$ and $N_f=400$. This means that already for $N_f=100$ frequencies the best possible error is achieved for the SL formulation, given the underlying time-stepping scheme. In contrast, for the $\eta=1$ scheme the convolution quadrature approximation introduces errors that are larger than the underlying time-stepping rule for $N_f=100$, while again for $N_f=400$ the error of the time-stepping scheme seems to dominate.

\begin{figure}[h]
\centering \includegraphics[width=0.5\textheight]{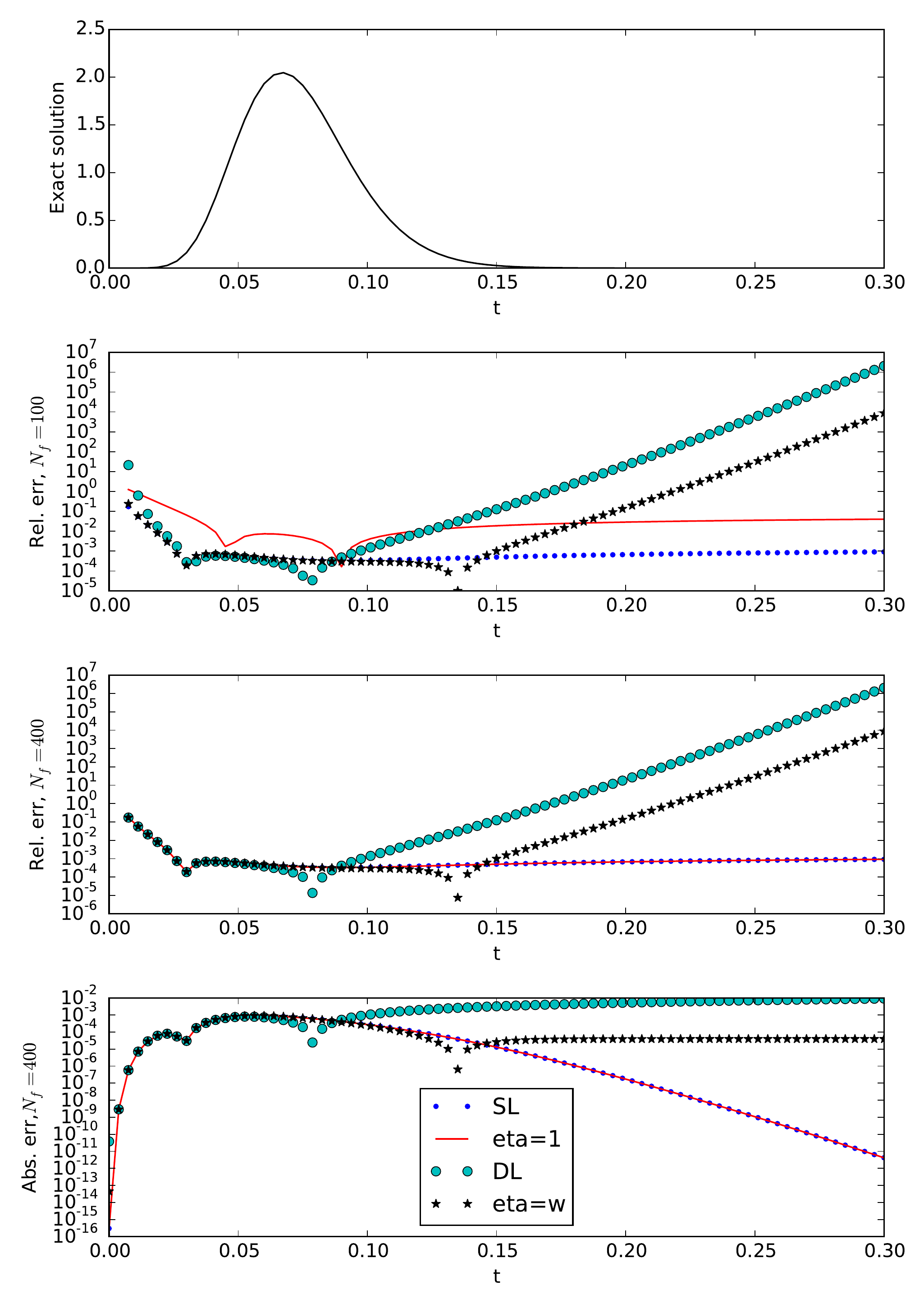}
\caption{Top: Analytic solution for $r=1.1$ and various time steps. Second figure: Relative error of various schemes for $N_f=100$ frequencies. Third Figure: The same as before but with $N_f=400$ frequencies. Bottom plot: Absolute error for $N_f=400$ frequencies.}
\label{figCompareIEform}
\end{figure}

\section{Conclusion}
\label{sec:conclusions}
Convolution quadrature methods have become a popular tool to solve
wave propagation problems in unbounded domains. In this paper we have
shown how the convergence of convolution quadrature methods depends on
the location of the poles of the underlying solution operator. It
therefore makes a significant difference whether we use a first kind,
second kind or combined integral equation formulation. The numerical
convergence results together with the comparison to the
analytical solution in Figure \ref{figCompareIEform} demonstrate the
importance of the location of the poles of the solution
operator. Indeed, the results in this paper are only a first step to
fully understand the influence of the poles of the frequency problems
on the numerical approximation of the time-domain solution.

An interesting aspect of these results is that although for a purely theoretical analysis only the scattering poles of the solution operators are relevant, in practice we need to reduce the exterior domain onto a problem on a finite domain, either by using a boundary integral formulation, or by introducing a PML layer. Both lead to additional poles that usually dominate the rate of convergence, as we have discussed in the case of a boundary integral equation formulation.

An import practical conclusion from the results in this paper is that it may be useful to overresolve in the frequency domain by computing more frequency solutions than there are time steps. This is important if the overall error is dominated not by the underlying time-stepping rule, but by the convolution quadrature approximation as depicted in the comparison of the error results for $N_f=100$ and $N_f=400$ in Figure \ref{figCompareIEform}.

Extensions of this current work to Maxwell problems are currently
under investigation. Finally, all results in this paper have been computed using the freely available boundary element library BEM++ (\url{www.bempp.org}), which provides a Python based interface to solve Laplace, Helmholtz and Maxwell boundary integral formulations. A time-domain toolbox for BEM++ is in planning.

\section*{Acknowledgements}
The authors would like to thank Lehel Banjai from Heriot-Watt University Edinburgh and Peter Monk and Francisco Javier-Sayas, both in Delaware, for introducing them to this topic and for many fruitful discussions on convolution quadrature methods and the results of this paper.

\appendix
\section{The 2-stages Runge-Kutta Radau IIa scheme}
\label{appendix_rungekutta}
A Runge-Kutta scheme can be described by its Butcher tableau of the form
\begin{align}
\begin{array}{c|cccc}
c_1    & a_{1,1} & a_{1,2}& \dots & a_{1,m}\\
c_2    & a_{2,1} & a_{2,2}& \dots & a_{2,m}\\
\vdots & \vdots & \vdots& \ddots& \vdots\\
c_m    & a_{m,1} & a_{m,2}& \dots & a_{m,m} \\
\hline
       & b_1    & b_2   & \dots & b_m
\end{array},
\end{align}
where $b,c \in \mathbb{R}^m$ and $A \in \mathbb{R}^{m \times m}$, with $m$ the number of stages, and by
\begin{align}
\Delta(z) = \left( A + \frac{z}{1 -z} \mathbbm{1} b^t\right)^{-1}
\end{align}
with $\mathbbm{1} = \left(1, \ldots, 1\right)^{t} \in \mathbb{R}^m$.
The $2$-stages Radau IIa scheme of order three is defined by the
tableau of the form
\begin{align}
\begin{array}{c|cc}
1/3  & 5/12  & -1/12  \\
1 & 3/4  &  1/4 \\
\hline
    & 3/4 & 1/4 \\
\end{array}.
\end{align}

We can diagonalize $\Delta(z)$ explicitly when $z \neq 3 \sqrt{3} - 5$:
\begin{align}
\Delta(z)  = \mathbb{P}(z) \mathbb{D}(z) \mathbb{P}^{-1}(z) \quad \text{ with } \mathbb{D}(z) = \mbox{diag}\left( \gamma_1(z), \gamma_2(z)\right) \notag
\end{align}
and
\begin{align}
\gamma_1(z) & = 2 + z - \sqrt{-2 + 10z + z^2} \notag \\
\gamma_2(z) & = 2 + z + \sqrt{-2 + 10z + z^2}
\end{align}
In case of higher order Radau IIa schemes it is not possible to obtain
an explicit diagonalization for each $z$ and the eigenvalues, and
eigenvectors need to be approximated numerically.



\bibliographystyle{siam}
\bibliography{cq_exponential2}
\end{document}